\documentclass[12pt]{amsart}
\usepackage[all]{xy} 		
\usepackage{longtable}
\usepackage{tikz}
\usetikzlibrary{arrows}
\usepackage{amsxtra,amsmath,amsfonts,amscd, amssymb, mathrsfs} 
\usepackage{tikz-cd}
\usepackage{mathtools}
\usepackage{tabu}
\usepackage{subcaption}
\usepackage{float}
\usepackage{color}
\usepackage{mystyle}
\usepackage[T1]{fontenc}
\usepackage{float}
\usepackage{thmtools}
\usepackage{thm-restate}

\DeclarePairedDelimiter\floor{\lfloor}{\rfloor}

\DeclareTextFontCommand{\emph}{\boldmath\bfseries}

\numberwithin{equation}{section}

\newcommand{\hi}{{\widehat{1}}}
\newcommand{\ho}{{\widehat{0}}}
\newcommand{\ov}[1]{\overline{#1}}
\newcommand{\wt}[1]{\widetilde{#1}}

\newcommand{\tsigma}{\widetilde{\sigma}}
\newcommand{\ses}{\widetilde{\Sigma}}

\newcommand{\Pyr}{\mathrm{Pyr}}
\newcommand{\Bary}{\mathrm{Bary}}

\newcommand{\cB}{\mathcal{B}}
\newcommand{\cP}{\mathcal{P}}
\newcommand{\cS}{\mathcal{S}}

\newcommand{\id}{\mathrm{id}}
\newcommand{\bK}{\mathbf{K}}

\definecolor{cof}{RGB}{219,144,71}
\definecolor{pur}{RGB}{186,146,162}
\definecolor{greeo}{RGB}{91,173,69}
\definecolor{greet}{RGB}{52,111,72}

\DeclareMathOperator{\rk}{rank}

\title[Poset subdivisions and the mixed $cd$-index]{Poset subdivisions and the mixed $cd$-index}

\author{Patrick Dornian}
\address{Patrick Dornian}
\email{patrick.dornian@gmail.com}

\author{Eric Katz}
\address{Eric Katz, Department of Mathematics, The Ohio State University University, 231 W. 18th Ave., Columbus, OH 43210}
\email{katz.60@osu.edu}

\author{Ling Hei Tsang}
\address{Ling Hei Tsang, Department of Mathematics, The Ohio State University University, 231 W. 18th Ave., Columbus, OH 43210}
\email{tsang.79@osu.edu}

\begin{document}

\begin{abstract}
The $cd$-index is an invariant of Eulerian posets expressed as a polynomial in noncommuting variables $c$ and $d$. It determines another invariant, the $h$-polynomial. In this paper, we study the relative setting, that of subdivisions of posets. We introduce the mixed $cd$-index, an invariant of strong formal subdivisions of posets, which determines the mixed $h$-polynomial introduced by the second author with Stapledon. The mixed $cd$-index is a polynomial in noncommuting variables $c',d',c,d$, and $e$ and is defined in terms of the local $cd$-index of Karu. Here, use is made of the decomposition theorem for the $cd$-index. We extend the proof of the decomposition theorem, originally due to Ehrenborg--Karu, to the class of strong formal subdivisions. We also compute the mixed $cd$-index in a number of examples.
\end{abstract}

\maketitle

\section{Introduction}

The enumeration of faces of polytopes is a central question in geometric combinatorics. One can count the number of faces or of flags, but there are many relations stemming from Euler's formula applied to faces and links of faces. Consequently, finding a packaging of the face and flag numbers is of key importance. Two approaches have emerged: the $h$-polynomial and the $cd$-index. The natural setting for each of these is not polytopes but rather Eulerian posets. These are graded posets such that any interval of positive length has an equal number of elements of each parity.

The $h$-polynomial in the case of simplicial polytopes emerges from counting the number of faces in each dimension, writing a generating function, and applying a linear change of variables. The redundancies are then reflected in the symmetry of the $h$-polynomial. The $h$-polynomial was generalized to Eulerian posets through a recursive definition inspired by intersection homology of toric varieties.

The $cd$-index is defined by first counting the number of flags of each rank set and encoding them in generating polynomial in noncommuting variables $a$ and $b$. As a consequence of the generalized Dehn--Sommerville equations \cite{bayer85}, this generating polynomial can be written more compactly as a polynomial in variables $c=a+b$ and $d=ab+ba$ \cite{bayer91}. See \cite{bayer19} for a survey. While the $cd$-index is not non-negative for every Eulerian poset, as shown in \cite[Exercise~3.193(a)]{stanley12}, it has desirable positivity properties for certain classes of posets \cite{stanley94,karu06,ehrenborg07}.

A natural relative extension of the study of polytopes is that of subdivisions of polytopes. Here, the natural notion of subdivision for an Eulerian poset is a strong formal subdivision by a lower Eulerian poset. 
A strong formal subdivision is a poset map $\phi\colon\Gamma\to B$ satisfying certain properties.
By work of the second author with Stapledon \cite{katz16}, one can enrich the $h$-polynomial to a multivariable analogue, the \emph{mixed $h$-polynomial}, which captures enumerative properties of the subdivision. The definition is motivated by the mixed Hodge theory of degenerating families of toric varieties (see \cite[Sec.~6]{katz16} for a discussion). The definition follows a certain pattern. One expresses enumerative invariants of $\Gamma$ in terms of those of $B$ and of $\Gamma_x=\phi^{-1}([\ho_B,x])$ for $x\in B$; this has the form
\[	h(\Gamma;t,u)= \sum_{ x \in B }  \ell^h_{[\ho_B,x]}(\Gamma_x;t,u) \cdot g([x,\hi_B];t,u)\]
where $h$ and $g$ are the (toric) $h$- and $g$-polynomials, and $\ell^h$ is the  \emph{local $h$-polynomial}, an invariant of subdivisions introduced by Stanley \cite{stanley92}. Such a formula is inspired by the decomposition theorem \cite{BBD}. To define an invariant of subdivisions,  one introduces new variables $t',u'$ to distinguish terms arising from invariants of $\Gamma_x$ and from invariants of $B$ and defines
	\[
		h_B(\Gamma;t',u',t,u) =  \sum_{x \in B } (t'u')^{\rk(\Gamma_x)} \cdot 
		\ell^h_{[\ho_B,x]}(\Gamma_x;t/t',u/u') \cdot g([x,\hi_B];tt',uu').
	\]
This invariant, by its nature, specializes to invariants of $\Gamma$ but also captures properties of the subdivision, often with desirable positivity properties. 

The purpose of this paper is to introduce the mixed $cd$-index $\Omega_{\phi}$ of a strong formal subdivision $\phi\colon\Gamma\to B$. We build on work by Karu \cite{karu06} and Ehrenborg--Karu \cite{ehrenborg07}. 
Karu defined in \cite{karu06} the \emph{local $cd$-index} $\ell^\Phi$
of a near-Gorenstein$^*$ posets. 
Using sheaf cohomology, Ehrenborg and Karu proved a decomposition theorem for the $cd$-index
for the class of strong Gorenstein$^*$ subdivisions  \cite{ehrenborg07}.
We extend the definition of the local $cd$-index to the class of near-Eulerian posets \cite{stanley94}, which are defined to be those obtained from an Eulerian poset by removing the unique maximal element $\hi$ and some element covered by $\hi$. We give a combinatorial proof of  the decomposition theorem for strong formal subdivisions of rank $0$:
\vspace{-1mm}
\begin{restatable}[Decomposition theorem for the $cd$-index]{thm}{restate}\label{thm:decomp}
Let $\Gamma$ be a lower Eulerian poset and $B$ be an Eulerian poset. 
	Let $\phi \colon  \Gamma \rightarrow B$ be a strong formal subdivision of rank $0$. 
	Then
	\[\Phi_{\Gamma} = \sum_{x \in B} \ell^\Phi_{\Gamma_x} \cdot \Phi_{[x, \hi]}. \]
\end{restatable}
\noindent We hope that this proof might be of independent interest.
%A strong formal subdivision is the weakest condition imposed on a poset map where the local $cd$-index $\ell^\Phi_{\Gamma_x}$ is defined, due to the near-Eulerian characterization for poset maps (See Proposition~\ref{p:nearcri}), hence we prove the theorem in its maximal generality.
%We also prove a version of the decomposition theorem for the local $cd$-index (see Proposition~\ref{p:localdecomp}) that is analogous to
%a decomposition formula for the local $h$-polynomial \cite[Lemma~4.7]{katz16}
For a strong formal subdivision $\phi \colon \Gamma\to B$ of rank $0$, we define the \emph{mixed $cd$-index} $\Omega_{\phi}$ as 
	\[\Omega_{\phi}(c',d',c,d, e) := \sum_{x \in B} \ell^\Phi_{
			\Gamma_x}(c',d') \cdot \Phi_{[x, \hi]}(c,d, e).
	\]
The mixed $cd$-index specializes to the $cd$-index of $\Gamma$.
While our approach to the mixed $cd$-index proceeds by analogy to the construction of the mixed $h$-polynomial, we hope to find a sheaf theoretic description and establish some positivity properties in future work. 

As an invariant of posets, the $cd$-index completely determines the $h$-polynomial \cite{bayer00}. As invariants of subdivisions, neither the local $cd$-index nor the $cd$-index is strong enough to determine even the local $h$-polynomial (see Example~\ref{exam:threesubs}).
Fortunately, the mixed $cd$-index does determine both the local $h$-polynomial and the mixed $h$-polynomial. Here, we follow the approach of Bayer--Ehrenborg \cite{bayer00} who rephrase the construction of the $cd$-index in terms of a coalgebra $R_\Phi$. In this case, we define a comodule $R_\Omega$ over $R_\Phi^e$ (where $R^e_\Phi$ is a slight modification of $R_\Phi$). The mixed $cd$-index, which is an element of $R_\Omega$ satisfies a structural equation (\ref{t:selfre}) phrased in terms of this comodule structure. Using this structural equation, we define two linear maps $L_\Omega$ and $H'_\Omega$ recursively and prove the following theorem:
\begin{theorem}\label{thm:mixedmaps}
The linear maps $L_\Omega\colon R_\Omega\rightarrow \bK[t,u]$, $H'_\Omega\colon R_{\Omega}\rightarrow \bK[t, t', u, u']$ send the mixed $cd$-index $\Omega$ to the local $h$-polynomial and the mixed $h$-polynomial respectively.
\end{theorem}

Finally, we compute some examples of the mixed $cd$-index.
We discuss the different behaviors for the $h$-polynomial and the $cd$-index in Example~\ref{exam:additive}.
Then, we give a recursive formula of the local $cd$-index of the barycentric subdivision of a simplex and write down the mixed $cd$-indices for the barycentric subdivisions of the Boolean algebras $\cB_2$ and $\cB_3$.
The {\em cutting operation} is applied to compute the mixed $cd$-index of some polytopal subdivisions.

The paper is organized as follows. In Section~\ref{sec:posets}, we set the notation for posets and study near-Eulerian posets. Section~\ref{sec:cdindex} reviews the $cd$-index and gives the definition of the local $cd$-index for a near-Eulerian poset. In Section~\ref{sec:sfs}, we study properties of strong formal subdivisions of rank $0$, setting the stage for the proof of the decomposition theorem for the $cd$-index in Section~\ref{sec:decomp}.
The definition and structural property of the mixed $cd$-index are given in Section~\ref{sec:mixcd}. 
In Section~\ref{sec:hpolys}, we construct a linear map that takes the mixed $cd$-index to the mixed $h$-polynomial. 
In Section~\ref{sec:examples}, we compute the mixed $cd$-index of some examples. 

\medskip
\noindent
{\it Acknowledgements.} 
We would like to thank Kalle Karu, Satoshi Murai and Hailun Zheng for valuable conversations.
Sections~\ref{sec:cdindex}, \ref{sec:sfs} and \ref{sec:decomp} are from the first author's Master thesis \cite{dornian16}.
The second author was partially supported by NSF DMS 1748837. The third author was partially supported by the Croucher Foundation.

\section{Posets}\label{sec:posets}

In this section we give a brief introduction to posets. See \cite[Chapter~3]{stanley12} for more details.

Let $P$ be a poset. A \emph{chain} (alt. \emph{flag}) of $P$ is a totally ordered subset of $P$:
\[C = \{x_0 < x_1 < \dots < x_k\}.\]
For elements $s, t \in P$ such that $s<t$, we say that $t$ \emph{covers} $s$ if there exists no $u \in P$ such that $s < u < t$.
We say that a chain is \emph{maximal} if it is not properly contained in another chain.

\begin{defn}
	Given a poset $P$ and $s,t \in P$, we define the \emph{intervals} of $P$ as follows:
	{\renewcommand{\arraystretch}{1.4}
	\begin{center}
$\begin{tabu}{l c l }
[s,t] & := & \{x  \in P \, | \, s \leq x \leq t \}; 
\\\relax
[s,t) & := & \{x \in P \, | \, s \leq x < t \}; \\\relax
[s,\infty) & := & \{x\in P \, | \, s\leq x\}.
 \relax
\end{tabu}$
	\end{center}}
\end{defn}

We use $\widehat{0}$ and $\widehat{1}$ to denote the unique minimal element and the unique maximal element of $P$ if they exist.

Given a poset $P$ of rank $n$, we can \emph{adjoin} an element $x$ to $P$, by adding $x$ to the underlying set and providing relations for $x$. 
We write $I<x$ if we set $y<x$ whenever $y\in I$. 
If we are adjoining a maximal element $\hi$ of rank $n+1$, we do not have to provide relations, since it is understood that we set $P<\hi$. We write $\ov{P}$ for the poset $P\cup \{\hi\}$ for simplicity.

\begin{defn}
For a finite poset $P$, it is said to be \emph{graded} if every maximal chain has the same length. In this case we call the length of maximal chains the \emph{rank} of $P$, denoted by $\rk(P)$.
A \emph{ranked poset} is a pair $(P, \rho_P)$ such that $P$ is a graded poset, $\rho_P \colon P\to \mathbb{Z}$ is a function
and $\rho_P(y) - \rho_P(x) = \rk([x, y])$ for any $x\leq y$.
\end{defn}

We call $\rho_P$ the \emph{rank function} of $P$.
By abuse of notation we write $\rho$ when the poset is clear. 

If a graded poset contains $\ho$, it is naturally a ranked poset with the \emph{natural rank function}, given by sending an element $x$ to the rank of the interval $[\ho, x]$. In fact, every rank function of a graded poset is the sum of the natural rank function and an integer.

\begin{notn}
From now on every poset is graded with $\ho$.
Unless otherwise specified, we are using the natural rank function, i.e. $\rho(\ho) = 0$.
\end{notn}

If $P$ contains $\ho$, a subposet $I$ is said to be a \emph{lower order ideal} of $P$ if for any $x\in I$ we have $[\ho, x]\subset I$.
The lower order ideal generated by $S\subset P$ is defined to be the smallest lower order ideal in $P$ containing $S$. 
If the elements in $S\neq \emptyset$ are of the same rank, then the lower order ideal generated by $S$ is a graded poset with $\ho$.

\begin{defn}
	For a graded poset $P$ of rank $n$ with $\ho$, we define the \emph{boundary} of $P$ as the lower order ideal generated by the rank $n-1$ elements covered by exactly one element.
\end{defn}

In particular if $P$ has $\hi$, then every rank $n-1$ element is covered by exactly one element, namely $\hi$. Thus 
$\partial P = P\setminus \{\hi \}$ and $\ov{\partial P} = P = \partial \ov{P}$. For a polytope, the boundary of its face poset is the face poset of its boundary.

\begin{defn}
We say that a graded poset $P$ with $\ho$ and $\hi$ is \emph{Eulerian} if every interval of positive rank satisfies the Euler--Poincar\'{e} relation. That is, it has the same number of elements of each parity. 
We say that a  graded poset is \emph{locally Eulerian} if every interval is Eulerian.
We say that a graded poset is \emph{lower Eulerian} if it has $\ho$ and is locally Eulerian.
\end{defn}

Examples of Eulerian posets include face lattices of polytopes.

\begin{example} \label{exam:boolean}
	The subsets of $[n]:=\{1, \dots , n\}$ ordered by inclusion is known as the \emph{Boolean algebra} $\cB_n$. The natural rank function is given by taking the size of the subset.
	It is an Eulerian lattice of rank $n$. It is also the face lattice of the $(n-1)$-simplex.
\end{example}

\begin{defn}
For graded posets $P$ and $Q$ with $\ho$ and $\hi$, we define the \emph{join} of $P$ and $Q$ to be the poset $P * Q$ with underlying set $(P\setminus \{\hi_P\}) \cup (Q\setminus \{\ho_Q\})$ where we have the relation $x\leq y$ if one of the following conditions is satisfied:
\begin{enumerate}
\item $x\leq y$ in $P$;

\item $x\leq y$ in $Q$; or

\item $x\in P$ and $y\in Q$.
\end{enumerate}
\end{defn}

Observe that the join of Eulerian posets of positive rank is Eulerian.

\subsection{Near-Eulerian posets}

We discuss near-Eulerian posets, which were introduced by Stanley \cite{stanley94}.

\begin{defn}
		We say a rank $n$ poset $P$ is \emph{near-Eulerian} if there exists some rank $n+1$ Eulerian poset $Q$ such that $P = Q \setminus \{q, \hi\}$ for some $q$ of rank $n$.
\end{defn}

It can be seen that if such Eulerian poset $Q$ exists, it is necessarily unique.
We call $Q$ the \emph{associated Eulerian poset} of $P$.

\begin{prop}\label{p:nearEulboundisboundEul}
The boundary of a near-Eulerian poset is the boundary of an Eulerian poset.
\end{prop}

\begin{proof}
This follows from the observation that for a near-Eulerian poset $P$, its boundary is given by $\partial P = [\ho, q)\subset Q$, where $Q = P \cup \{\hi, q\}$ is the associated Eulerian poset, and the fact that the interval $[\ho, q]$ is Eulerian.
\end{proof}

\begin{defn}
	Given a graded poset $P$ of rank $n$, define the \emph{semisuspension} of $P$ as $\ses P:= P\cup \{q\}$ where we adjoin a new element $q$ of rank $n$ and set $\partial P < q$.
\end{defn}

\begin{prop}\label{p:sesnear}
Let $P$ be a graded poset of rank $n$. Then $P$ is near-Eulerian if and only if the semisuspension $\ses P$ is the boundary of an Eulerian poset. 
\end{prop}

\begin{proof}
If $P$ is near-Eulerian, there exists some Eulerian poset $Q$ such that $P = Q \setminus \{q, \hi\}$ where $q$ is a rank $n$ element. Note that the half-open interval $[\ho, q)$ in $Q$ is the same as the boundary of $P$, thus taking semisuspension gives $\ses{P} = Q \setminus \{\hi\} = \partial Q$, which is the boundary of an Eulerian poset. 

On the other hand, suppose $\ses P$ is the boundary of some Eulerian poset $Q$. Then $Q:=\ov{\ses P}$ is Eulerian and if we remove the newly adjoined elements $q$ and $\hi$ from $Q$ we get back $P$, which means $P$ is near-Eulerian.
\end{proof}

It is necessary for a near-Eulerian poset to be lower Eulerian and have the same number of elements of each parity. However, there are posets satisfying those conditions that are not near-Eulerian.

We have the following characterization for near-Eulerian posets.

\begin{prop}[Near-Eulerian criterion for posets] \label{p:charnear} Let $P$ be a lower Eulerian poset of rank $n$.
	Then $P$ is near-Eulerian if and only if $n>0$ and the following conditions are satisfied for some graded lower order ideal $I$ of rank $n-1$:
	\begin{enumerate}
		\item For $x\in I$, we have 		
\begin{equation}\label{eq:c1}
\sum_{y\in I \cap [x, \infty)} (-1)^{\rho(y)} = (-1)^{n+1} \text{, and}
\end{equation}
\begin{equation}\label{eq:c2}
   \sum_{y\in [x, \infty)} (-1)^{\rho(y)} = 0;
\end{equation}

		\item For $x \notin I$, we have
		\begin{equation}\label{eq:c3}
		      \sum_{y\in [x, \infty)} (-1)^{\rho(y)} = (-1)^{n}.		
		\end{equation}
	\end{enumerate}
	Furthermore if the conditions for $I$ hold, then we have $\partial P = I$.
\end{prop}

\begin{proof}
	Suppose $P$ is near-Eulerian. 
	Let $Q$ be the associated Eulerian poset of $P$ such that $Q \setminus \{q, \hi\} = P$. Let $I = \partial P$. Clearly $I$ is a graded lower order ideal of rank $n-1$.
For $x\in \partial P$, the intervals $[x, q]$ and $[x, \hi]$ are Eulerian, hence we have
\[		
\sum_{{y\in [x, q)}} (-1)^{\rho(y)} + (-1)^{\rho(q)}= 0
		\text{ and }
\sum_{\substack{y\in [x, \hi]\\ y \neq q, \hi}} (-1)^{\rho(y)} + (-1)^{\rho(q)} + (-1)^{\rho(\hi)} = 0,
\]
which are just Condition~\eqref{eq:c1} and Condition~\eqref{eq:c2}, since $[x, q)$ in $Q$ is the same as $I \cap [x, \infty)$ in $P$, and $[x, \hi]\setminus \{q, \hi\}$ in $Q$ is the same as 
$[x, \infty)$ in $P$.
Similarly for $x\notin \partial P$, by considering the Eulerian poset $[x, \hi]$ we have
\[\sum_{y\in [x, \hi)} (-1)^{\rho(y)} +  (-1)^{\rho(\hi)} = 0,
\]
which gives Condition~\eqref{eq:c3}.

\medskip

	Now suppose the conditions are satisfied for some graded lower order ideal $I$ of rank $n-1$
	 We define $Q$ by adding $q$ of rank $n$ and $\hi$ to $P$ and setting $I<q$. Since $I$ is graded of rank $n-1$, the resulting poset $Q$ is graded of rank $n+1$.
It suffices to show that $Q$ is Eulerian. Since $P$ is lower Eulerian, we need only show that intervals of the form $[x, \hi]$ and $[x, q]$ have the same number of elements of each parity. 
But as we have seen they follow from Condition~\eqref{eq:c1}, Condition~\eqref{eq:c2} and Condition~\eqref{eq:c3}. Therefore $Q$ is Eulerian and $P$ is near-Eulerian. Since $[\ho, q)$ in $Q$ is given by $\partial P$, we have that $I = \partial P$.
\end{proof}

\begin{corollary}\label{p:smallernear}
Suppose $P$ is a near-Eulerian poset. Then the subposet $[x, \infty)$ is near-Eulerian if $x\in \partial P$, and is the boundary of an Eulerian poset if $x\notin \partial P$.
\end{corollary}

We introduce an important class of near-Eulerian posets.

\begin{prop}
	Eulerian posets of positive rank are near-Eulerian.
\end{prop}

\begin{proof}
Let $P$ be an Eulerian poset of rank $n$. 
The poset $Q:= \ov{\ses{P}}$ is the join of the Eulerian posets $P$ and $\cB_2$, hence $Q$ is also Eulerian.
By Proposition~\ref{p:sesnear} we have that $P$ is near-Eulerian.
\end{proof}

A natural example of a near-Eulerian poset is the face poset of the boundary of a polytope with exactly one facet removed. We may think of the semisuspension as ``capping'' the polytope off with the missing facet.

\section{The \texorpdfstring{$cd$}{cd}-index and the local \texorpdfstring{$cd$}{cd}-index}\label{sec:cdindex}

\begin{notn}
For graded posets with $\ho$, we modify the definition of chain to mandate that every chain contains $\ho$.
\end{notn}

For a chain $C = \{ \ho = x_0 < x_1 < \dots < x_k \}$, the \emph{rank set} $\rho(C)$ is defined to be $\{\rho(x_1), \dots, \rho(x_k)\}\subseteq [n]= \{1, 2, \dots, n\}$. 
Note that we do not include $\rho(\ho)$ in $\rho(C)$, since the presence of $\ho$ is assumed.

\begin{defn}
	The \emph{flag $f$-vector} of $P$ is defined to be the function  $\alpha_P\colon  2^{[n]} \rightarrow \mathbb{Z}_{\geq 0}$
	where $\alpha_P (S)$ is the number of chains containing $\ho$ with rank set $S$.
	The \emph{flag $h$-vector} of $P$ is defined to be the function $\beta_P\colon  2^{[n]} \rightarrow \mathbb{Z}$ where
	\[\beta_P(S) = \sum_{T\subseteq S} (-1)^{\#(S \setminus T)} \alpha_P(T).\]
\end{defn}

By inclusion-exclusion we have $\alpha(S) = \sum_{T \subseteq S} \beta(T).$

Let $R_{\Psi_{}} := \bK \langle a, b \rangle$ be the polynomial ring generated by non-commutative variables $a$ and $b$, where $\bK$ is a field of characteristic $0$.
For $S\subset [n]$, we define the characteristic monomial $u_S = u_1 u_2 \cdots u_n \in R_{\Psi_{}}$ by letting
\[u_i = \begin{cases}
		a & \text{if $i \notin S$} \\
		b & \text{if $i \in S$}
	\end{cases}\]
For example, if $n = 5$ and $S = \{1, 2, 5\}$ we have $u_S = bbaab$.

\begin{defn}
Let $P$ be a graded poset of rank $n$ with $\ho$.
The \emph{$ab$-index} of $P$ is the element in $R_{\Psi_{}}$ defined by
	\[\Psi_{P} (a,b) = \sum_{S \subseteq [n]} \beta_P (S) u_S \, , \]
	and the \emph{flag enumerator} of $P$ is the element in $R_{\Psi_{}}$ defined by
	\[\Upsilon_P (a, b) = \sum_{S \subseteq [n]} \alpha_P (S) u_S. \]
\end{defn}

Alternatively the flag enumerator can be given as the sum of the characteristic monomials of all the chains in $P$. Since $P$ is graded, both the $ab$-index and the flag enumerator are homogeneous of degree $n$.

It is easily seen that the flag enumerator and the $ab$-index are equivalent by a linear change of variables \cite{stanley94}:
\begin{align*}
	\Upsilon_P(a,b) & = \Psi_P(a + b, b),   \\
	\Psi_P(a, b)    & =\Upsilon_P(a - b,b).
\end{align*}

Note that our definition of flag enumerator would count $\hi$ if the poset contains the maximal element.
That is, if $P$ contains $\hi$ then $\Psi_{P} = \Psi_{\partial P} \cdot a$.
This convention is slightly different from the one in \cite{stanley12}.

Bayer and Billera \cite{bayer85} showed that the flag $f$-vector satisfies the generalized Dehn--Sommerville relations for Eulerian posets. Fine observed that this is equivalent to the following additional constraints on the $ab$-index \cite{bayer91}.

\begin{theorem}\label{t:firstcd} Let $P$ be an Eulerian poset. Then there exists a polynomial in non-commuting variables $c$ and $d$ denoted $\Phi_{\partial P}(c,d)$ that satisfies $\Phi_{\partial P}(a + b, ab + ba) = \Psi_{\partial P}(a,b)$.
\end{theorem}

Thus with $c=a+b$ and $d=ab+ba$, we let $R_\Phi :=\bK \langle c, d\rangle \subset R_{\Psi_{}}$ be a subring of $R_{\Psi_{}}$. 
We call the polynomial $\Phi_{\partial P}(c, d)\in R_\Phi$ the \emph{$cd$-index} of $\partial P$. 
An immediate consequence of the theorem is that the $ab$-index of $\partial P$ is symmetric under switching $a$ and $b$.
The symmetry of the $ab$-index is analogous to the symmetry of the $h$-vector of simplicial polytopes and spheres.

\subsection{The local \texorpdfstring{$cd$}{cd}-index}

We may decompose the $ab$-index of a near-Eulerian poset $P$ into a term that is expressible in $c$ and $d$, and a remainder term.
This result is described in \cite{karu06, ehrenborg07}. 

\begin{lemma}[Local $cd$-index]\label{l:localcd}
	Given a near-Eulerian poset $P$ with boundary $\partial P$, we may write its $ab$-index as

	\[\Psi_{P} = \ell^\Psi_{P}\ + \Psi_{\partial P} \cdot a \, , \]
	where $\ell^\Psi_{P} := \Psi_{\ses P} - \Psi_{\partial P} \cdot (a + b)$. Furthermore $\ell^\Psi_{P}$ is $cd$-expressible.
\end{lemma}

Note that if a polynomial in variables $a$ and $b$ can be expressed as $f(c, d) + g(c,d )\cdot a$ then such $f$ and $g$ are unique. 

\begin{proof}
	The semisuspension $\ses P:= P\cup \{q\}$ contains chains of exactly two types: those contained in $P$, and chains in $\partial P$ followed by $q$. 
Rewriting the flag enumerator of $\ses P$, we have
\[
\Upsilon_{	P} = \Upsilon_{\ses{P}} - \Upsilon_{\partial P} \cdot b.
\]
	By substituting $a - b$ for $a$, we have
	\begin{align*}
		\Psi_{P} & = \Psi_{ \ses P} - \Psi_{\partial P} \cdot b                                     \\
		         & = (\Psi_{ \ses P} - \Psi_{\partial P} \cdot (a + b)) + \Psi_{\partial P} \cdot a \\
		         & = \ell^\Psi_{P} +  \Psi_{\partial P} \cdot a.
	\end{align*}
	Since both $\ses P$ and $\partial P$ are the boundaries of Eulerian posets, the polynomial $\ell^\Psi_{P}$ is $cd$-expressible and we have
	\[\ell^\Phi_{P} = \Phi_{\ses{P}} - \Phi_{\partial P} \cdot c. \qedhere
	\]
\end{proof}

By convention we set the local $cd$-index of the single element poset $\{\widehat{0} \}$ as $1$. 

We call $\ell^\Psi_{P}(a,b)$ and $\ell^\Phi_{P}(c,d)$ the \emph{local $ab$-index} and \emph{local $cd$-index} of $P$ respectively. 
We also define the \emph{local flag enumerator} as $\ell^\Upsilon_P := 
\ell^\Psi_{P} (a + b, b)$ .

We have defined $cd$-index for boundaries of Eulerian posets. We can also define $cd$-index for near-Eulerian posets {\cite[Lemma~3.1]{karu06}}.

\begin{defn}\label{def:localcdfornear}
	Let $P$ be a near-Eulerian poset of rank $n$. Then we define the \emph{$cd$-index} of $P$ to be
	\[\Phi_{P} := \ell^\Phi_{P} + \Phi_{\partial P}.\]
\end{defn}

Note that this is a non-homogeneous polynomial. The polynomial $\ell^\Phi_{P}$ has degree $n$
while the polynomial $\Phi_{\partial P}$ has degree $n-1$, where $\deg(c)=1$ and $\deg(d)=2$.

The local $cd$-index measures how different a near-Eulerian poset $P$ is from the Eulerian poset $\ov{\partial P}$. Thus for Eulerian posets we have the following proposition.

\begin{prop} \label{p:localcdvanishing}
	If $P$ is Eulerian of positive rank, then the local $cd$-index $\ell^\Phi_{P} = 0$.
\end{prop}

\begin{proof}
	This follows from the observation $\Psi_P = \Psi_{\partial P} \cdot a$.
\end{proof}

Note that for an Eulerian poset $P$ of positive rank, because of $\Psi_P = \Psi_{\partial P} \cdot a$, we have $\Phi_{P} = \Phi_{\partial P}$.

\begin{remark}
We have defined the $cd$-index for two classes of posets, namely 
the class of the boundaries of Eulerian posets and the class of near-Eulerian posets. The $cd$-index is well-defined, since no poset is both the boundary of an Eulerian poset and a near-Eulerian poset. Even though the singleton poset is the boundary of an Eulerian poset and is Eulerian, it is not a near-Eulerian poset.
Note that by Proposition~\ref{p:nearEulboundisboundEul}, the $cd$-index is also defined for the boundary of a near-Eulerian poset.
\end{remark}

\section{Strong formal subdivisions}\label{sec:sfs}

A function $\phi\colon \Gamma \rightarrow B$ is said to be \emph{order-preserving} if $y_1\leq y_2$ in $\Gamma$ implies $\phi(y_1)\leq \phi(y_2)$.
If both $\Gamma$ and $B$ are graded posets with rank functions $\rho_\Gamma$ and $\rho_B$, a function $\phi\colon \Gamma \rightarrow B$ is said to be \emph{rank-increasing} if $\rho_\Gamma(y)\leq \rho_B(\phi(y))$ for all $y\in \Gamma$.

If both $\Gamma$ and $B$ are graded with $\ho$ and 
$\phi\colon \Gamma \rightarrow B$ is order-preserving, rank-increasing and surjective, then we have $\phi(\ho_\Gamma) = \ho_B$ and $\rho_\Gamma(\ho_\Gamma) \leq \rho_B(\ho_B)$.
We define the \emph{rank} of $\phi$ as $\rk(\phi) = \rho_B(\ho_B) - \rho_\Gamma(\ho_\Gamma)$.

For an order-preserving, rank-increasing and surjective function $\phi\colon \Gamma\to B$, we define a few preimage posets for $x\in B$ and $y\in \Gamma$:
	{\renewcommand{\arraystretch}{1.4}
	\begin{center}
$\begin{tabu}{l c l }
\Gamma_x & := & \{y' \in \Gamma \, | \, \phi(y')\leq x\}; \\\relax
\Gamma_{\geq y} & := & \{y' \in \Gamma \, | \, y\leq y'\}; \\\relax
(\Gamma_{\geq y})_x & := & \{ y'\in \Gamma \, | \,  y\leq y', \phi(y') \leq x \}.
 \relax
\end{tabu}$	\end{center}}

We recall a notion of subdivision studied by Katz--Stapledon in \cite{katz16}.

\begin{defn}
	Let $\Gamma$ and $B$ be locally Eulerian posets. 
	An order-preserving, rank-increasing and surjective function $\phi\colon \Gamma \rightarrow B$ is said to be a \emph{strong formal subdivision} if the following are true:
	\begin{enumerate}
		\item (Strongly surjective) For all $y\in \Gamma$ and $x\in B$ such that $\phi(y)\leq x$, there exists $y' \geq y$ such that $\phi(y') = x$ and $\rho(y') = \rho(x)$; and
		
		\item For all $y\in \Gamma$ and $x\in P$ such that $\phi(y)\leq x$,
		      \begin{equation}\label{eq:sfsdef}
		      \sum_{\substack{y'\in \phi^{-1}(x)\\
		      y\leq y'}} (-1)^{\rho(y')} = (-1)^{\rho(x)}.
		      \end{equation}
	\end{enumerate}
\end{defn}

There are several nice properties of the strong formal subdivision as shown in \cite{katz16}. 
The composition of strong formal subdivisions is a strong formal subdivision. 
For any $y\in \Gamma$, the restriction $\phi|_{\Gamma_{\geq y}}
\colon \Gamma_{\geq y}\rightarrow B_{\geq \phi(y)}$ is a strong formal subdivision.
For any lower order ideal $I$ of $B$, its preimage under $\phi$ is locally Eulerian and the restriction $\phi|_{\phi^{-1}(I)}\colon \phi^{-1}(I)\to I$ is a strong formal subdivision. 
Combining these last two properties, the restriction $\phi_{(\Gamma_{\geq y})_x}\colon (\Gamma_{\geq y})_x\rightarrow [\phi(y), x]$ is also a strong formal subdivision. 

We cite a characterization of strong formal subdivisions %proved in the same paper 
\cite[Lemma~3.18]{katz16}. 

\begin{lemma}\label{l:charSFS}
	Let $\phi\colon \Gamma\rightarrow B$ be an 
order-preserving, rank-increasing and strongly surjective function	 between locally Eulerian posets. Then $\phi$ is a strong formal subdivision if and only if for all $y\in \Gamma$, $x\in B$ such that $\phi(y)\leq x$ we have 
	\begin{equation}\label{eq:sfs}
	\sum_{y'\in (\Gamma_{\geq y})_x} (-1)^{\rho(y')} =
		\begin{cases}
			(-1)^{\rho(x)} & \text{if } \phi(y) = x \\
			0       & \text{otherwise}.
		\end{cases}
	\end{equation}	
\end{lemma}

\begin{notn}
From now on we let $\Pi$ and $\Gamma$ be lower Eulerian posets and $B$ be an Eulerian poset of rank $n$. We also let $\sigma \colon \Pi\rightarrow \Gamma$ and $\phi\colon \Gamma\rightarrow B$ be order-preserving, rank-increasing and surjective functions of rank $0$. Thus without loss of generality we may assume the rank functions are natural rank functions, i.e.~$\rho(\ho) = 0$.
\end{notn}

We have the following characterization of near-Eulerian posets for strong formal subdivision of rank $0$.

\begin{prop}[Near-Eulerian criterion for poset maps] \label{p:nearcri}
	Let $\Gamma$ be lower Eulerian and $B$ be Eulerian, both of rank $n$.
	Let $\phi\colon \Gamma \rightarrow B$ be an order-preserving, rank-increasing and surjective function of rank $0$.
	Then $\phi$ is a strong formal subdivision if and only if for any $\ho \neq x\in B$, the preimage poset $\phi^{-1}[\ho, x]$ is near-Eulerian of rank $\rho(x)$ with boundary $\phi^{-1}[\ho, x)$.
\end{prop}

\begin{proof}
First we see that the function $\phi$ is strongly surjective if and only if for any $x\in B$, the preimage poset $\phi^{-1}[\ho, x]$ is graded of rank exactly $\rho(x)$. 

Now suppose $\phi$ is a strong formal subdivision. We want to show that for $\ho \neq x\in B$ the preimage poset $\Gamma_x = \phi^{-1}[\ho, x]$ is near-Eulerian with boundary $\phi^{-1}[\ho, x)$. 
If $B$ is the singleton poset this is clear since $\phi$ is of rank $0$.
Since $[\ho, x]$ is a lower order ideal, the preimage poset $\Gamma_x$ is lower Eulerian and the restriction $\phi_x\colon \Gamma_x\to [\ho, x]$ is also a strong formal subdivision of rank $0$.
Thus by induction on the rank of $B$, it suffices to only prove 
$\Gamma=\phi^{-1} (B)$ is near-Eulerian with boundary $\phi^{-1}(\partial B)$.

We use the near-Eulerian criterion for posets with the lower order ideal $I :=  \phi^{-1}(\partial B)$.
Both Condition~\eqref{eq:c2} and Condition~\eqref{eq:c3} follow from Lemma~\ref{l:charSFS} with $x=\hi$ and $y \in \Gamma$ since
\[\sum_{y' \in [y, \infty)} (-1)^{\rho(y')} =
 \sum_{y'\in (\Gamma_{\geq y})_{\hi}} (-1)^{\rho(y')}
\]
and $y \in I$ if and only if $\phi(y) \neq \hi$.
For $y\in I$, Condition~\eqref{eq:c1} follows from
Lemma~\ref{l:charSFS} with $x=\hi$ and an application of Condition~\eqref{eq:c2}:
\begin{align*}
\sum_{y' \in I\cap [y, \infty)} (-1)^{\rho(y')} &=
\sum_{y' \in I\cap [y, \infty)} (-1)^{\rho(y')} - \sum_{y' \in [y, \infty)} (-1)^{\rho(y')}\\
& = - \sum_{\substack{y' \\ y\leq y', \phi(y') = \hi}} (-1)^{\rho(y')}\\
& = -\sum_{y'\in (\Gamma_{\geq y})_{\hi}} (-1)^{\rho(y')}\\
& =
(-1)^{n+1},
\end{align*}
since $B = \partial B \cup \{\hi\}$ and hence $y'\notin \phi^{-1}(\partial B)$ implies $\phi(y') = \hi$.
Thus by the near-Eulerian criterion for posets, the poset $\Gamma$ is near-Eulerian with boundary $\phi^{-1}(\partial B)$.

\medskip

	Now suppose $\phi^{-1}[\ho, x]$ is near-Eulerian with boundary $\phi^{-1}[\ho, x)$ for $0\neq x\in B$.
	For $x\in B$ and $y\in \Gamma_x$ we want to prove \eqref{eq:sfs} in Lemma~\ref{l:charSFS}:	
\[		\sum_{y'\in (\Gamma_{\geq y})_x} (-1)^{\rho(y')} =
		\begin{cases}
			(-1)^{\rho(x)} & \text{if } \phi(y) = x \\
			0       & \text{otherwise}.
		\end{cases}\]
	Since $\phi(y) = x$ if and only if $y\notin \phi^{-1}[\ho, x)$, the equation follows from Condition~\eqref{eq:c2} and Condition~\eqref{eq:c3} in the statement of the near-Eulerian criterion for posets.
\end{proof}

\begin{remark}
We can extend the criterion to subdivisions of positive rank by further requiring the preimage $\phi^{-1}\{\ho\}$ to be the boundary of an Eulerian poset.
\end{remark}

For a subdivision on the interior of a poset, we have the following corollary. 

\begin{corollary}\label{c:intsub}
Let $\phi\colon \Gamma\rightarrow B$ be an order-preserving, rank-increasing and surjective function of rank $0$. 
Suppose further $\phi^{-1}(\partial B) = \partial \Gamma$ and the restriction to the boundary $\phi|_{\partial \Gamma} \colon \partial \Gamma\rightarrow \partial B$ is an isomorphism. Then $\phi$ is a strong formal subdivision if and only if $\Gamma$ is near-Eulerian.
\end{corollary}

\begin{proof}
If $\phi$ is a strong formal subdivision, then by the near-Eulerian criterion for poset maps $\phi^{-1}(B) = \Gamma$ is near-Eulerian.

On the other hand, in order to prove that $\phi$ is a strong formal subdivision, by the near-Eulerian criterion for poset maps, it suffices to prove that for any $x\neq \ho$ the preimage poset $\phi^{-1}[\ho, x]$ is near-Eulerian with boundary $\phi^{-1}[\ho, x)$. For $x\neq \hi$, we have $x\in\partial B$, and this is automatic.  The case of $x=\hi$ is equivalent to showing $\Gamma$ has boundary $\phi^{-1}(\partial B)$, which is one of our assumptions.
%Since $\phi|_{\partial \Gamma}$ is a strong formal subdivision, it remains to prove that $\Gamma$ has boundary $\phi^{-1}(\partial B)$, which is one of our assumptions.
\end{proof}

For a near-Eulerian poset $\Gamma$, we can define a function $\phi\colon \Gamma \to \ov{\partial \Gamma}$ that is the identity on $\partial \Gamma$ and sends $\Gamma\setminus \partial \Gamma$ to $\hi \in \ov{\partial \Gamma}$. By the corollary, the function $\phi$ is a strong formal subdivision.

\subsection{Extensions}
We explore some basic properties of strong formal subdivisions.

\begin{defn}[Extension of posets and poset maps]\label{def:extension}
Let $\sigma\colon \Pi \rightarrow \Gamma$ be an order-preserving, rank-increasing and surjective function of rank $0$ between lower Eulerian posets.
Suppose $\Gamma$ is a lower order ideal in some Eulerian poset $B$.
We define the extension of $\Pi$ over $B$ to be
$\wt{\Pi} := \Pi \cup B\setminus \Gamma$ with relations $y \leq x$ when
	\begin{enumerate}
		\item $y \leq x$ in $B$ for $x, y \in B\setminus \Gamma$;
		
		\item $y\leq x$ in $\Pi$ for $x, y\in \Pi$; or
		
		\item \label{c:mixposet} $\sigma(y) \leq x$ in $B$ for $y\in \Pi$, $x\in B$.
	\end{enumerate}
	We also define the extension of $\sigma$ to be $\tsigma\colon \wt{\Pi} \rightarrow B$ with $\tsigma(y) = \sigma(y)$ if $y\in \Pi\subset \wt{\Pi}$ and $\tsigma(y) = y$ if $y\in B\setminus \Gamma\subset \wt{\Pi}$.
\end{defn}

The poset $\wt{\Pi}$ is a graded poset of the same rank as $B$ and its natural rank function is given by
\[\rho(y) =	\begin{cases}
		\rho_{\Pi} (y) & \text{ if } y \in \Pi \\
		\rho_{B} (y)     & \text{ if }  y \in B.
	\end{cases}\]
	Furthermore the function $\tsigma$ is order-preserving, rank-increasing and surjective of rank $0$.
We call the induced function $\wt{\sigma}$ the extension of $\sigma$ over $B$ and the induced poset $\wt{\Pi}$ the extension of $\Pi$ over $B$.

\begin{prop}\label{p:extension}
Suppose $\sigma\colon \Pi\rightarrow \Gamma\subset B$ is a strong formal subdivision of rank $0$. Then $\wt{\Pi}$, the extension of $\Pi$ over $B$, is lower Eulerian and $\tsigma$ is a strong formal subdivision. Moreover $\wt{\Pi}$ is Eulerian if $\Gamma \subsetneq B$.
\end{prop}

\begin{proof}
Let $[y_1, y_2]$ be an interval in $\wt{\Pi}$ where $y_1\neq y_2$. We want to show that it has the same number of elements of each parity.
If it is entirely in $B\setminus \Gamma$ or $\Pi$, then we are done.
Otherwise, we have $y_1\in \Pi$, $y_2\in B$ and $\sigma(y_1)\neq y_2$. 
Then we have
\begin{align*}
\sum_{z\in [y_1, y_2]\subset \wt{\Pi}} (-1)^{\rho(z)} 
& =
\sum_{z'\in [\sigma(y_1), y_2]} \sum_{\substack{z\in \tsigma^{-1}(z')\\ y_1\leq z}} (-1)^{\rho(z)} \\ 
& =
\sum_{z'\in [\sigma(y_1), y_2]\cap \Gamma} \sum_{\substack{z\in \tsigma^{-1}(z')\\ y_1\leq z}} (-1)^{\rho(z)} +
\sum_{z'\in [\sigma(y_1), y_2]\setminus \Gamma} \sum_{\substack{z\in \tsigma^{-1}(z')\\ y_1\leq z}} (-1)^{\rho(z)}\\
& =
\sum_{z'\in [\sigma(y_1), y_2]\cap \Gamma} (-1)^{\rho(z')} +
\sum_{z'\in [\sigma(y_1), y_2]\setminus \Gamma} (-1)^{\rho(z')}\\
& = \sum_{z'\in [\sigma(y_1), y_2]} (-1)^{\rho(z')}\\
& = 0,
\end{align*}
since for $z'\notin \Gamma$ the only element of $\tsigma^{-1}(z')$ is $z'$ and for $z'\in \Gamma$, we may apply \eqref{eq:sfsdef} in the definition of a strong formal subdivision.
Thus the interval $[y_1, y_2]$ has the same number of elements of each parity, and $\wt{\Pi}$ is lower Eulerian.

%\medskip
Now we want to prove that the function $\tsigma$ is a strong formal subdivision. 
By construction the function is order-preserving, rank-increasing and surjective.
For strong surjectivity, it suffices to prove that for any $x\in B$ the preimage $\tsigma^{-1}[\ho, x]$ is graded of rank $\rho_B(x)$.
For $x\in \Gamma$ the condition follows from the strong surjectivity of $\sigma$ while for $x\in B\setminus \Gamma$ the condition follows from the construction of $\tsigma$.

For \eqref{eq:sfsdef}, it suffices to prove 
\[
\sum_{\substack{y'\in \tsigma^{-1}(x)\\
		      y\leq y'}} (-1)^{\rho(y')} = (-1)^{\rho(x)}\]
for $x\in \Gamma$ and $y \in \Pi$ such that $\tsigma(y)\leq x$, since for $x\in B\setminus \Gamma$, the preimage $\tsigma^{-1}(x)$ is a singleton of rank $\rho(x)$. 
But this follows immediately from $\sigma$ being a strong formal subdivision.

If $\Gamma\subsetneq B$, then $\hi_B$ is not in $\Gamma$ and $\hi_B \in \wt{\Pi}$. Hence $\wt{\Pi}$ is Eulerian.
\end{proof}

If $\Gamma$ happens to be near-Eulerian, it is naturally a lower order ideal of some Eulerian poset, namely $\ov{\ses{\Gamma}}$. This leads to the following proposition.

\begin{prop}\label{p:sfsses}
Let $\sigma\colon \Pi\to \Gamma$ be an order-preserving, rank-increasing and surjective function of rank $0$.
Suppose further that $\Gamma$ is near-Eulerian.
Then $\Pi$ is near-Eulerian and 
$\sigma$ extends to a strong formal subdivision between the semisuspensions of $\Pi$ and $\Gamma$.
Furthermore $\sigma^{-1}(\partial \Gamma) = \partial \Pi$.
\end{prop}

\begin{proof}
Since $\Gamma$ is near-Eulerian, we have that $\Gamma\subset \ov{\ses{\Gamma}}$ as a lower order ideal. Then $\tsigma$, the extension over $\ov{\ses{\Gamma}}$, is a strong formal subdivision from $\wt{\Pi}$ to $\ov{\ses{\Gamma}}$.
By Proposition~\ref{p:extension}, the poset $\wt{\Pi}$ is Eulerian since $\Gamma\subsetneq \ov{\ses{\Gamma}}$. By the construction of $\wt{\Pi}$, we have $\wt{\Pi} = \Pi \cup \{q_\Pi, \hi\}$, which means $\Pi$ is near-Eulerian and $\wt{\Pi} = \ov{\ses{\Pi}}$.

Now by the near-Eulerian criterion for poset maps, 
we have that the boundary of $[\ho, q_\Pi] = \tsigma^{-1}[\ho, q_{\Gamma}]$
is given by 
$\partial \Pi = [\ho, q_\Pi) = \tsigma^{-1}[\ho, q_{\Gamma}) = \sigma^{-1}(\partial \Gamma	)$. 
\end{proof}

\subsection{Intermediate maps}

Let $\phi\colon \Gamma\to B$ be an order-preserving, rank-increasing and surjective function of rank $0$, where $\Gamma$ is lower Eulerian and $B$ is Eulerian, both of rank $n$.
Let $x_0, \dots, x_N$ be a list of all elements in $B$ such that 
\[\rho(x_0) \leq \rho(x_1) \leq \dots \leq \rho(x_N)
.\] In particular we have $x_0=\ho$ and $x_N = \hi$. Note that for any $i$ the subset $\{x_0, \dots, x_i\}$ is a lower order ideal of $B$.

\begin{defn}
	We define the \emph{intermediate posets} $B_i$ for $0 \leq i \leq N$ as follows:
	\begin{itemize}
		\item $B_i = \phi^{-1}\{x_0, \dots, x_i\} \cup \{ x_{i+1}, \dots, x_N\}$.
		\item We say $y \leq x$ in $B_i$ if and only if exactly one of the following holds:
		      \begin{enumerate}
			      \item  $x, y \in \{ x_{i+1}, \dots, x_N\}$ and $y \leq x$ in $B$;
			      \item  $x, y \in \phi^{-1}\{x_0, \dots, x_i\}$ and  $y \leq x$ in $\Gamma$; or
			      \item $y \in \phi^{-1}\{x_0, \dots, x_i\}, x \in \{ x_{i+1}, \dots, x_N\}$ and $\phi(y) \leq x$ in $B$.
		      \end{enumerate}
	\end{itemize}
We define the \emph{intermediate maps} $\phi_i\colon  B_{i} \rightarrow B_{i-1}$ for $0 < i \leq N$ by 
\[
	\phi_i(y) =
	\begin{cases}
		\phi(y) & \text{if } y \in \phi^{-1}(x_i)  \\
		y       & \text{otherwise.}
	\end{cases}
\]
\end{defn}

\medskip

Note that while we are writing $\Gamma = B_N$ for simplicity, the poset $\Gamma$ is not Eulerian in general. For $0<i<N$, the natural rank function of $B_i$ is given by
\[\rho_{B_i}(y) =
	\begin{cases}
		\rho_{\Gamma} (y) & \text{ if } y \in \phi^{-1}\{x_0, \dots, x_i\} \\
		\rho_{B} (y)     & \text{ if }  y \in \{ x_{i+1}, \dots, x_N\}.
	\end{cases}
\]

By directly checking the definition, one can see that the intermediate maps $\phi_i$'s are order-preserving, rank-increasing and surjective.
Thus we factorize $\phi\colon \Gamma\rightarrow B$ into poset maps:
\[\begin{tikzcd}
\Gamma = B_N \arrow{r}{\phi_N} &  B_{N-1} \arrow{r}{\phi_{N-1}} & \dots \arrow{r}{\phi_2}&  B_1 \arrow{r}{\phi_1} & B_0 = B.
\end{tikzcd}\]
Note that for fixed $i$, the map $\phi_i$ is the identity map if and only if $\phi^{-1}(x_i)$ is a singleton.

\begin{prop}
Let $\phi$ be a strong formal subdivision. Then for $0 < i <N$ the poset $B_i$ is Eulerian and for $0 < i \leq N$ the intermediate map $\phi_i$ is a strong formal subdivision.
\end{prop}

\begin{proof}
For simplicity we fix $i>0$ and write $x = x_i$ and $\tau = \phi_i$.

Since $\phi$ is a strong formal subdivision, by the near-Eulerian criterion for poset maps $\Gamma_x$ is near-Eulerian. 
Now consider the restriction of $\tau_x\colon \Gamma_x\to [\ho, x]\subset B_{i-1}$. 
By construction, the poset map $\tau_x$ is the identity on the boundary $\partial \Gamma_x$.
Thus by Corollary~\ref{c:intsub} we have that $\tau_x$ is a strong formal subdivision.

By looking at the definitions of $B_i$ and $\phi_i$, one can see that they are the extension of $\Gamma_x$ over $B_{i-1}$ and the extension of $\tau_x$ over $B_{i-1}$, respectively. Thus by the properties of extensions, the poset $B_i$ is lower Eulerian and $\phi_i$ is a strong formal subdivision.
In particular for $0<i<N$, since $[\ho, x]\subsetneq B_{i-1}$, each $B_i$ contains $\hi$ and hence is Eulerian.
\end{proof}

\begin{example}
	Let $B$ be a tetrahedron and let $\Gamma$ be obtained from $B$ by performing stellar subdivision on an edge of $B$. The subdivision $\phi\colon \Gamma \to B$ is a strong formal subdivision of rank $0$, and can be factorized into smaller strong formal subdivisions as follows. Here we are omitting the identity maps.
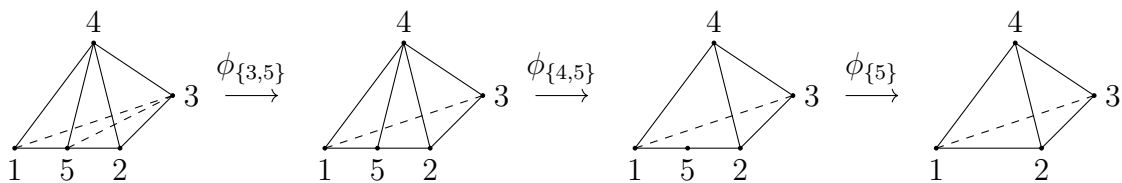
\begin{figure}[H]
	\centering
		\begin{tikzpicture}[scale=0.7]
			\draw (0,0)--(2,0)--(3,1)--(1.5,2);
			\draw (0,0)--(1.5,2)--(2,0);
			\draw (1,0)--(1.5,2);
			\draw [dashed] (0,0) --(3,1);
			\draw [dashed] (1,0)--(3,1);
			\draw [fill] (0,0) circle [radius = 1pt];
			\draw [fill] (2,0) circle [radius = 1pt];
			\draw [fill] (3,1) circle [radius = 1pt];
			\draw [fill] (1.5,2) circle [radius = 1pt];
			\draw [fill] (1,0) circle [radius = 1pt];
			\draw (0,0) node[anchor=north]{$1$};
			\draw (1,0) node[anchor=north]{$5$};
			\draw (2,0) node[anchor=north]{$2$};
			\draw (3,1) node[anchor=west]{$3$};
			\draw (1.5,2) node[anchor=south]{$4$};
			\draw [->] (4,1) -- (5,1) node[midway, above] {$\phi_{\{3,5\}}$};
		\end{tikzpicture}
		\begin{tikzpicture}[scale=0.7]
			\draw (0,0)--(2,0)--(3,1)--(1.5,2);
			\draw (0,0)--(1.5,2)--(2,0);
			\draw (1,0)--(1.5,2);
			\draw [dashed] (0,0) --(3,1);
			\draw [fill] (0,0) circle [radius = 1pt];
			\draw [fill] (2,0) circle [radius = 1pt];
			\draw [fill] (3,1) circle [radius = 1pt];
			\draw [fill] (1.5,2) circle [radius = 1pt];
			\draw [fill] (1,0) circle [radius = 1pt];
			\draw (0,0) node[anchor=north]{$1$};
			\draw (1,0) node[anchor=north]{$5$};
			\draw (2,0) node[anchor=north]{$2$};
			\draw (3,1) node[anchor=west]{$3$};
			\draw (1.5,2) node[anchor=south]{$4$};
			\draw [->] (4,1) -- (5,1) node[midway, above] {$\phi_{\{4,5\}}$};
		\end{tikzpicture}		
		\begin{tikzpicture}[scale=0.7]
			\draw (0,0)--(2,0)--(3,1)--(1.5,2);
			\draw (0,0)--(1.5,2)--(2,0);
			\draw [dashed] (0,0) --(3,1);
			\draw [fill] (0,0) circle [radius = 1pt];
			\draw [fill] (2,0) circle [radius = 1pt];
			\draw [fill] (3,1) circle [radius = 1pt];
			\draw [fill] (1.5,2) circle [radius = 1pt];
			\draw [fill] (1,0) circle [radius = 1pt];
			\draw (1,0) node[anchor=north]{$5$};
			\draw (0,0) node[anchor=north]{$1$};
			\draw (2,0) node[anchor=north]{$2$};
			\draw (3,1) node[anchor=west]{$3$};
			\draw (1.5,2) node[anchor=south]{$4$};
			\draw [->] (4,1) -- (5,1) node[midway, above] {$\phi_{\{5\}}$};
		\end{tikzpicture}
		\begin{tikzpicture}[scale=0.7]
			\draw (0,0)--(2,0)--(3,1)--(1.5,2);
			\draw (0,0)--(1.5,2)--(2,0);
			\draw [dashed] (0,0) --(3,1);
			\draw [fill] (0,0) circle [radius = 1pt];
			\draw [fill] (2,0) circle [radius = 1pt];
			\draw [fill] (3,1) circle [radius = 1pt];
			\draw [fill] (1.5,2) circle [radius = 1pt];
			\draw (0,0) node[anchor=north]{$1$};
			\draw (2,0) node[anchor=north]{$2$};
			\draw (3,1) node[anchor=west]{$3$};
			\draw (1.5,2) node[anchor=south]{$4$};
		\end{tikzpicture}
		\caption{An example of factorizing strong formal subdivisions.}
	\end{figure}
\end{example}

\section{The Decomposition theorem}\label{sec:decomp}

In this section we prove Ehrenborg and Karu's decomposition theorem for the $cd$-index \cite[Theorem~2.7]{ehrenborg07}.
By counting flags, we are able to extend the theorem to the class of strong formal subdivisions. 

\restate*

Recall that, by convention, the local $cd$-index of the singleton poset is given by $1$. Hence, the summand for $x = \ho$ is given by $1\cdot \Phi_{B}$, the $cd$-index of $B$. 
We note that if $\Gamma$ is Eulerian, then $\ell^\Phi_\Gamma = 0$ and $\Phi_\Gamma$ is homogeneous of degree $n-1$.

\begin{proof} 
Using the $cd$-indices of the intermediate posets, we have
\begin{equation}\label{eq:telescopebreakdown}
\Phi_\Gamma = (\Phi_\Gamma - \Phi_{B_{N-1}}) + (\Phi_{B_{N-1}}- \Phi_{B_{N-2}}) +  \dots + (\Phi_{B_1} - \Phi_B) + \Phi_B.
\end{equation}
We show that the first summand in the equation is given by the local $cd$-index of $\Gamma$.
Since $B_{N-1} = \phi^{-1}\{x_1, \dots, x_{N-1}\} \cup \{\hi\}$ is Eulerian, its boundary is given by $\partial B_{N-1} = \phi^{-1}\{x_1, \dots, x_{N-1}\} = 
\phi^{-1}(\partial B)\subset \Gamma$.
On the other hand, by the near-Eulerian criterion for poset maps, we have $\phi^{-1}(\partial B) = \partial \Gamma$. Thus, $\partial \Gamma =\phi^{-1}(\partial B)=\partial B_{N-1}$ and, from the definition of the $cd$-index of a near-Eulerian poset, the first summand in \eqref{eq:telescopebreakdown} gives the local $cd$-index of $\Gamma$:
\[
\Phi_\Gamma - \Phi_{B_{N-1}} = (\ell^\Phi_\Gamma + \Phi_{\partial \Gamma}) -  \Phi_{\partial B_{N-1}} = \ell^\Phi_\Gamma.
\]
Thus, it remains to prove for $0< i < N$, 
\[
\ell^\Phi_{\Gamma_{x_i}} \cdot \Phi_{[x_i, \hi]}= \Phi_{B_i} - \Phi_{B_{i-1}}.
\]
Since $\Phi_{P} = \Phi_{\partial P}$ for every Eulerian poset $P$ of positive rank, it suffices to show, as we will do below, the following equation for $0< i < N$
	\begin{equation} \label{eq:telescope}
		\ell^\Upsilon_{\Gamma_{x_i}} \cdot \Upsilon_{[x_i, \widehat{1})}= \Upsilon_{\partial B_i} - \Upsilon_{\partial B_{i-1}}.	\qedhere
	\end{equation}
\end{proof}

Equation~\eqref{eq:telescope} will follow from the two lemmas below, the first of which follows from definitions.

\begin{lemma} \label{lem:flagcount1}
For $0<i<N$ we have the identity
	\[
		\Upsilon_{\partial B_i} - \Upsilon_{\partial B_{i-1}} = \sum_{\substack{C \text{ a chain of } \partial B_i \\ C \text{ contains some element in }\phi_i^{-1}(x_i)}} u_{\rho(C)}
 \, \, \, \, \, - \sum_{\substack{C \text{ a chain of } \partial B_{i-1} \\ C \text{ contains }x_i}} u_{\rho(C)}.\]
\end{lemma}

\begin{lemma} \label{lem:flagcount2}
For $0<i<N$ we have the identity
	\[
		\ell^\Upsilon_{\Gamma_{x_i}} \cdot
		\Upsilon_{[x_i, \hi)} = \sum_{\substack{C \text{ a chain of } \partial B_i \\ C \text{ contains some element in }\phi_i^{-1}(x_i)}} u_{\rho(C)}
 \, \, \, \, \, - \sum_{\substack{C \text{ a chain of } \partial B_{i-1} \\ C \text{ contains }x_i}} u_{\rho(C)}.\]
\end{lemma}

\begin{proof}
Note that the interval $[\ho, x_i]$ is considered to be in $B_{i-1}$.
	We first observe
	\begin{align*}
		\ell^\Upsilon_{\Gamma_{x_i}} & = \Upsilon_{\Gamma_{x_i}} - \Upsilon_{[\ho, x_i]}                                                   \\
		& = \sum_{C \text{ a chain of } \Gamma_{x_i}} u_{\rho(C)} 
		 \, \, \, \, \, 
		 - \sum_{C \text{ a chain of } [\ho, x_i]} u_{\rho(C)} \\
		& = \sum_{\substack{C = C_1 < y \text{ a chain of } \Gamma_{x_i}                                       \\ \phi_i(y) = x_i}} u_{\rho(C)}  
		\, \, \, \, \, 
		-
		\sum_{C=C_1 < x_i \text{ a chain of } [\ho, x_i]} u_{\rho(C)}.
	\end{align*}
	where the second equality follows from the definition of the local flag enumerator and the last equality follows from canceling chains that are present in both terms, namely the chains contained entirely in $\partial \Gamma_x$ in the first sum and chains whose maximal element is not $x_i$ in the second sum. 
	This is possible, since by Proposition~\ref{p:nearcri}, %the near-Eulerian criterion for poset maps, 
	we have $\partial \phi_i^{-1}[\ho, x_i] = \phi_i^{-1}[\ho, x_i)$. 
	Multiplying $\Upsilon_{[x_i, \hi)}$ gives the result, since every chain counted non-trivially in the difference of flag enumerators is the concatenation of a chain counted in $\ell^\Upsilon_{\Gamma_{x_i}}$ and a chain in $[x_i, \hi)$.
\end{proof}

Similarly we can extend the result to poset maps between near-Eulerian posets.

\begin{cor}
	Let $\sigma\colon  \Pi\to \Gamma$ be a strong formal subdivision of rank $0$ between near-Eulerian posets. Then we have the identity
	\[\Phi_{\Pi} =\sum_{x\in \Gamma} \ell^\Phi_{\Pi_x} \cdot \Phi_{[x, \infty)}. \qedhere
	\]
\end{cor}

\begin{proof}
By considering the extension of $\sigma$ 
over $\ov{ \ses{\Gamma}}$, we factorize the extension map $\wt{\sigma}$ into intermediate maps. 
By taking restrictions, we get a factorization of $\sigma$ into intermediate maps.
Hence by Lemma~\ref{lem:flagcount1} and Lemma~\ref{lem:flagcount2}
we prove ~\eqref{eq:telescopebreakdown}. \qedhere
\end{proof}

\begin{remark}
The decomposition theorem holds for subdivisions of positive rank if for any Eulerian poset $B$, we define the local $cd$-index for $\partial B$ by setting $\ell^\Phi_{\partial B} := \Phi_{\partial B}$.
\end{remark}

Note that our results do not require the posets to be Cohen-Macaulay, which was needed in Ehrenborg--Karu's proof of the decomposition theorem for the $cd$-index.

Here is a decomposition result for the local $cd$-index. Note that this is analogous to the decomposition result for the local $h$-polynomial in \cite[Lemma~4.7]{katz16}.

\begin{prop}\label{p:localdecomp}
Let $\sigma\colon \Pi\rightarrow \Gamma$ be strong formal subdivisions of rank $0$ between near-Eulerian posets.
Then the local $cd$-index is given by
\[\ell^\Phi_{\Pi} =  \sum_{y\in \partial \Gamma} \ell^\Phi_{\Pi_y} \cdot 
\ell^\Phi_{[y, \infty)}
+ \sum_{y\in \notin \partial \Gamma} \ell^\Phi_{\Pi_y} \cdot \Phi_{[y, \infty)}.
\]
\end{prop}

Recall by Corollary~\ref{p:smallernear}, the interval $[y, \infty)$ in $\Gamma$ is the boundary of an Eulerian poset if $y\notin \partial \Gamma$ and is near-Eulerian if $y\in \partial \Gamma$.

\begin{proof}
The local $cd$-index $\ell^\Phi_{\Pi}$ is given by
$
\ell^\Phi_{\Pi} = 
\Phi_{\Pi} - \Phi_{\ov{\partial \Pi}}.
$
By Proposition~\ref{p:sfsses}, the strong formal subdivision $\sigma\colon \Pi\rightarrow \Gamma$ restricts to 
$\sigma|_{\partial \Pi}\colon \partial \Pi \to \partial \Gamma$. Thus the extension over $\ov{\partial \Gamma}$ gives 
$\wt{\sigma|_{\partial \Pi}}\colon \ov{\partial \Pi} \to \ov{\partial \Gamma}$ as a strong formal subdivision.
By applying the decomposition theorem to the strong formal subdivisions
$\sigma \colon \Pi\to \Gamma$ and $\wt{\sigma|_{\partial \Pi}} \colon {\ov{\partial \Pi}}
\to {\ov{\partial \Gamma}}$, 
we have
\begin{align*}
\Phi_{\Pi} - \Phi_{\ov{\partial \Pi}}
& =
\sum_{y\in \Gamma}\ell^\Phi_{\Gamma_y} \cdot \Phi_{[y, \infty)}
- 
\sum_{y\in \partial \Gamma} \ell^\Phi_{\Gamma_y} \cdot 
\Phi_{\ov{\partial [y, \infty)}}\\
& = \sum_{y\in \partial \Gamma} \ell^\Phi_{\Gamma_y} 
\cdot (\Phi_{[y, \infty)} - \Phi_{\ov{\partial [y, \infty)}})
+
\sum_{y\notin \partial \Gamma}
\ell^\Phi_{\Gamma_y} \cdot \Phi_{[y, \infty)}\\
& = \sum_{y\in \partial \Gamma} \ell^\Phi_{\Gamma_y} 
\cdot \ell^\Phi_{[y, \infty)}
+
\sum_{y\notin \partial \Gamma}
\ell^\Phi_{\Gamma_y} \cdot \Phi_{[y, \infty)}.
\end{align*}
Note that we use the fact that for $x\in \partial \Gamma$, we have $[x, \infty)\cap \partial \Gamma = \partial [x, \infty)$.
\end{proof}

In particular if we consider the composition of strong formal subdivisions $\sigma \colon \Pi \to \Gamma$ and $\phi \colon \Gamma \to B$, both of rank $0$, then for any $x\in B$ the restriction $\sigma_x \colon \Pi_x = \sigma^{-1}(\Gamma_x)\to \Gamma_x$ is a strong formal subdivision of rank $0$ between near-Eulerian posets, and the local $cd$-index of $\Pi_x$ is given by 
\[
\ell^\Phi_{\Pi_x} = 
\sum_{y\in \partial \Gamma_x} \ell^\Phi_{\Pi_y} \cdot \ell^\Phi_{\Gamma_{(\geq y)_x}}
+ \sum_{y\in \notin \partial \Gamma} \ell^\Phi_{\Pi_y} \cdot \Phi_{\Gamma_{(\geq y)_x}}.\]

\section{Properties of the mixed \texorpdfstring{$cd$}{cd}-index}\label{sec:mixcd}

In this section we define the mixed $cd$-index, an invariant of subdivisions. 

We introduce a new variable $e$ of degree $-1$ and let $R^e_{\Psi_{}}$ be the quotient of $\bK\langle a, b, e\rangle$ by the two-sided ideal generated by $e^2$, $ea$, $eb$, $ae$ and $be$. This is an algebra whose underlying vector space is $R_{\Psi_{}} \oplus \bK e$. Similarly, let $R^e_\Phi$ be the subalgebra of $R^e_{\Psi_{}}$ that has $R_\Phi \oplus \bK e$ as its underlying vector space.

We define $R'_{\Phi}$, the primed analogue of $R_\Phi$ by
$R'_{\Phi} := \bK \langle c', d'\rangle \subset \bK \langle a', b'\rangle$ with $c' = a'+ b'$ and $d' = a'b' + b' a'$. 
The \emph{mixed vector space} $R_{\Omega}:= R'_{\Phi} \otimes_{\bK} R^e_\Phi$ is the tensor product of two vector spaces. 
By abuse of notation we write $w' \cdot w$ for the element $w'\otimes w$ where $w' \in R'_{\Phi}$ is a primed monomial and $w\in R^e_\Phi$ is an unprimed monomial. 
We have a natural degree map on $R_{\Omega}$, where $\deg(a) = \deg(b) = \deg(a') = \deg(b') = 1$ and $\deg(e)  = -1$.

For the remainder of the paper, 
we refine the definition of the $cd$-index by setting the $cd$-index of the singleton poset to be $e\in R^e_\Phi$. Note that we recover the original $cd$-index by the specialization $e\mapsto 1$.

\begin{defn} Let $\phi\colon \Gamma\rightarrow B$ be a strong formal subdivision of rank $0$, where $\Gamma$ is lower Eulerian and $B$ is Eulerian.
The \emph{mixed $cd$-index} of $\phi\colon \Gamma\rightarrow B$ is an element in $R_{\Omega}$ given by 
	\[\Omega_{\phi}(c',d',c,d, e) = \sum_{x \in B} \ell^\Phi_{
			\Gamma_x}(c',d') \cdot \Phi_{[x, \hi]}(c,d, e). \]
\end{defn}

Note that the mixed $cd$-index is homogeneous of total degree $\rk(B) - 1$.

If both $\Gamma$ and $B$ are singleton posets, then the map $\phi\colon \Gamma\to B$ is an isomorphism, and the mixed $cd$-index is given by $\Omega_{\phi} = \ell^\Phi_{\{\ho\}} \cdot e = e$.

\begin{lemma}
	Under the specialization $c'\mapsto c$, $d'\mapsto d$ and $e\mapsto 1$, the mixed $cd$-index $\Omega_{\phi}$ specializes to the $cd$-index of $\Gamma$.
\end{lemma}

\begin{proof}
Under the specialization we have
\[\Omega_{\phi}(c,d,c,d, 1)=
\sum_{x \in B} \ell^\Phi_{
			\Gamma_x}(c,d) \cdot \Phi_{[x, \hi]}(c,d, 1),
\]
which gives the $cd$-index of $\Gamma$ by the decomposition theorem.
\end{proof}

\begin{lemma}\label{l:beginning}
	Under the specialization $c'\mapsto 0$, $d'\mapsto 0$ and $e\mapsto 1$, the mixed $cd$-index $\Omega_{\phi}$ specializes to the $cd$-index of $B$.
\end{lemma}

\begin{proof}
Under the specialization we have
\[\Omega_{\phi}(0,0,c,d, 1)=
\sum_{x \in B} \ell^\Phi_{
			\Gamma_x}(0,0) \cdot \Phi_{[x, \hi]}(c,d, 1) = 1\cdot \Phi_{[\ho, \hi]}(c,d),
\]
which gives the $cd$-index of $B$.
\end{proof}

There is an important structural property of the mixed $cd$-index, which we call the \emph{comodule property}. It is phrased in terms of the coalgebra structure of $R^e_\Phi$ and the comodule structure of $R_{\Omega}$.

Let us recall the definition of a coalgebra.
Fix a ground field $\bK$ of characteristic $0$. 
A triple $(C, \Delta, \epsilon)$ is a \emph{coalgebra} over $\bK$ if the following conditions are satisfied.
\begin{enumerate}
\item $C$ is a vector space over $\bK$;

\item $\Delta\colon C\to C\otimes C$ is a linear map satisfying the \emph{coassociativity condition} $(\Delta\otimes \id) \circ \Delta = (\id \otimes \Delta) \circ \Delta$; and

\item $\epsilon\colon C\to \bK$ is a linear map satisfying the \emph{counit condition} $(\epsilon \otimes \id) \circ \Delta = \id = (\id \otimes \epsilon) \circ \Delta$.
\end{enumerate}
We call $\Delta$ and $\epsilon$ the \emph{coproduct} and the \emph{counit} of $C$.

We present a variation of Ehrenborg and Fox's coalgebra structure on $R^e_{\Psi_{}}$ \cite{ehrenborg03}.
	Let $\Delta\colon R^e_{\Psi_{}} \rightarrow R^e_{\Psi_{}} \otimes R^e_{\Psi_{}}$ be the linear map defined by sending a monomial $w = w_1 \cdots w_n$ where each $w_i$ is a letter to 
		\[
	\Delta (w) = 
	\begin{cases}
	e\otimes w + \sum_{i=1}^n w_1\cdots w_{i-1} \otimes w_{i+1}\cdots w_n
	 + w\otimes e & \text{if } w\neq e\\
	e\otimes e & \text{if } w = e.
	\end{cases}\]
	Let $\epsilon \colon  R^e_{\Psi_{}} \rightarrow \bK$ be the characteristic function of $e\in R^e_{\Psi_{}}$. The following is a straight forward verification.

\begin{prop}
The triple $(R^e_{\Psi_{}}, \Delta, \epsilon)$ is a coalgebra over $\bK$.
\end{prop}

Note that $\Delta(c) = e\otimes c + 2 (1\otimes 1) + c\otimes e$ and $\Delta(d) = e\otimes d + c\otimes 1 + 1\otimes c + d\otimes e$. Hence the coproduct $\Delta$ restricts to $\Delta|_{R^e_\Phi}\colon R^e_\Phi \to R^e_\Phi \otimes R^e_\Phi$ and the triple $(R^e_\Phi, \Delta, \epsilon)$ is a coalgebra.

Ehrenborg and Readdy proved the following result in \cite{ehrenborg98}.

\begin{prop}\label{p:cdindexismorphism}
For any Eulerian poset $B$ we have
\[
\Delta (\Phi_{B}) = \sum_{\ho\leq x\leq \hi} \Phi_{[\ho, x]}
		\otimes \Phi_{[x, \hi]}.\]
\end{prop}

\begin{remark}
Let $\cP$ be the vector space over $\bK$ generated by isomorphism classes of Eulerian posets. Let $\Delta\colon \cP \to \cP \otimes \cP$ be the linear map defined by sending an Eulerian poset $B$ to 
		\[
	\Delta (B) = 
	\begin{cases}
	[\ho, \ho]\otimes B + \sum_{\ho <x< \hi} \, [\ho, x]\otimes [x, \hi]
	 + B\otimes [\hi, \hi] & \text{if } \rk(B)\neq 0\\
	[\ho, \ho]\otimes [\hi, \hi] & \text{if } \rk(B) = 0.
	\end{cases}\]
Let $\epsilon$ the characteristic function of the singleton poset. Then, the triple $(\cP, \Delta, \epsilon)$ forms a coalgebra.
Proposition~\ref{p:cdindexismorphism} is equivalent to saying that the $cd$-index is a coalgebra morphism from $(\cP, \Delta, \epsilon)$ to $(R^e_\Phi, \Delta, \epsilon)$.
\end{remark}

%%%%%%%%%%%%%%%%%%%%

Let us recall the definition of a comodule. 
Suppose $(C, \Delta, \epsilon)$ is a coalgebra over $\bK$. A pair $(M, \varrho)$ is a \emph{comodule} over $C$ if the following conditions are satisfied.
\begin{enumerate}
\item $M$ is a vector space over $\bK$;

\item $\varrho\colon M\to M\otimes C$ is a linear map satisfying the \emph{coassociativity condition} $(\id \otimes \Delta) \circ \varrho = (\varrho\otimes \id) \circ \varrho $; and

\item $\epsilon\colon C\to \bK$ from the coalgebra $(C, \Delta, \epsilon)$ satisfies the \emph{counit condition} $(\id \otimes \epsilon)\circ\varrho = \id$.
\end{enumerate}
We call $\varrho$ the \emph{comultiplication} of $M$.

We define a comodule structure on $R_{\Omega}$ as follows. 
Note that the $\bK$-vector space $R_{\Omega}$ is a left $R'_\Phi$-module. 
Hence the tensor product $R_{\Omega}\otimes R^e_\Phi$ is naturally a left $R'_\Phi$-module, where the scalar multiplication is defined by $w'\cdot (w_1 \otimes w_2) = (w' w_1)\otimes w_2$ for $w'\in R'_\Phi$ and $w_1\otimes w_2 \in R_{\Omega}\otimes R^e_\Phi$.
Now we define the linear map $\varrho\colon R_{\Omega} \rightarrow R_{\Omega} \otimes R^e_\Phi$ as follows:
For a primed monomial $w'$ and an unprimed monomial $w$, 
we set $\varrho(w' \cdot w) = w' \cdot \Delta (w)$ by considering $\Delta(w)$ an element of $R_{\Omega}\otimes R^e_\Phi$, and extend linearly.

\begin{prop}
The pair $(R_{\Omega}, \varrho)$ is a comodule over the coalgebra $(R^e_\Phi, \Delta, \epsilon)$.
\end{prop}

%\begin{proof}
%We want to show the coassociativity and the counit conditions are satisfied. The counit condition follows from the definitions of the counit $\epsilon$ and the extended comultiplication $\varrho$.
%For the coassociativity condition, it follows from the computation
%\[
%	(\id \otimes \Delta)( \varrho (w' \cdot  w))  = w' \cdot 
%		\sum w_{(1)}\otimes w_{(2)} \otimes w_{(3)} 
%		  = (\Delta \otimes \id)( \varrho (w' \cdot  w)),
%\]
%	since the extended coproduct $\Delta$ is coassociative.
%\end{proof}

\begin{proof}
This follows from verifying definitions.
\end{proof}

We prove the \emph{comodule property} of the mixed $cd$-index. 

\begin{theorem} \label{t:selfre}
Let $\phi\colon \Gamma\rightarrow B$ be a strong formal subdivision of rank $0$. Then 
	$$ \varrho(\Omega_{\phi})
	= \sum_{x\in B} \Omega_{\phi_x}
	\otimes \Phi_{[x,\hi]},$$
	where $\phi_x\colon \Gamma_x\to [\ho, x]$ is the restriction of $\phi$ to $\Gamma_x$.
\end{theorem}

\begin{proof}
By direct computation, we have
	\begin{align*}
\varrho(\Omega_{\phi}) 
		& =  \sum_{x\in B} \ell^\Phi_{\Gamma_x}(c', d') \cdot \Delta(\Phi_{[x, \hi]}(c, d, e))\\
		& =  \sum_{x\in B}  \ell^\Phi_{\Gamma_x}(c', d') \cdot \left(\sum_{y\geq x} \Phi_{[x,y]}(c, d, e) \otimes \Phi_{[y, \hi]}(c, d, e)\right)\\
		& =  \sum_{y\in B} \left(\sum_{x\leq y} \ell^\Phi_{\Gamma_x}(c', d') \cdot \Phi_{[x,y]}(c, d, e)\right) \otimes \Phi_{[y, \hi]}(c, d, e)\\		
		& =  \sum_{y\in B} \Omega_{\phi_y}
		\otimes \Phi_{[y, \hi]}(c, d, e).	
	\end{align*}
	By replacing $y$ with $x$, we get the desired expression.
\end{proof}

\begin{remark}
There is a subdivision analogue of the observation that the $cd$-index is a coalgebra homomorphism from $(\cP, \Delta, \epsilon)$ to $(R^e_\Phi, \Delta, \epsilon)$. We may consider all isomorphism classes of strong formal subdivisions, and let $\cS$ be the $\bK$-vector space generated by them. Then $\cS$ forms a comodule over $\cP$ by the map
\[\varrho\colon \cS\to\cS\otimes \cP\]
given by 
\[\varrho\left(\phi\colon \Gamma\to B\right)=\sum_{x\in B} \phi_x\otimes [x,\hi].\]
Then, the content of the above theorem is that $\Omega$ is a homomorphism of comodules.
\end{remark}

\section{The mixed \texorpdfstring{$cd$}{cd}-index and the mixed \texorpdfstring{$h$}{h}-polynomial}\label{sec:hpolys}

In this section, we review the mixed $h$-polynomial \cite{katz16} and show that it is determined by the mixed $cd$-index.

\subsection{The mixed \texorpdfstring{$h$}{h}-polynomial} 
Let $\mathbb{Z}[t,u]$ be the two-variable polynomial ring. We have an involution on $\mathbb{Z}[t,u]$ given by interchanging $t$ and $u$: for $f(t,u)\in \mathbb{Z}[t,u]$, $\overline{f}(t,u)=f(u,t)$. 
Note that involution preserves total degree of monomials. We will consider $\mathbb{Z}[t,u]$ as a ring graded by total degree. Here, our use of $\mathbb{Z}[t,u]$ differs from that of \cite{katz16} where $\mathbb{Z}[t]$ is used instead. The earlier definition arises by specializing $u$ to $1$.

\begin{defn} Let $B$ be an Eulerian poset of rank $n$. We will define the \emph{{$g$-polynomial}} $g(B;t,u)$ to be a particular element of $\mathbb{Z}[t,u]$ of total degree $n$. If $n= 0$, then $g(B;t,u)=1$. For $n>0$, then $g(B;t,u)$ is the unique polynomial of $t$-degree strictly less than $n/2$ satisfying
	\[\overline{g}(B;t,u)=\sum_{x\in B} g([\widehat{0},x];t,u) \cdot (t-u)^{n-\rk([\ho, x])}.\]
\end{defn}

Note that $g$-polynomial has total degree $n$.
For example, for the Boolean algebra $\cB_n$, we have $g(\cB_n; t, u) = u^n$.

\begin{defn} Let $\Gamma$ be a lower Eulerian poset of rank $n$. The \emph{$h$-polynomial} of $\Gamma$ is given by
	\[\overline{h}(\Gamma;t,u)=\sum_{x\in \Gamma} g([\widehat{0},x];t,u) \cdot (t-u)^{n-\rk([\ho, x])}.\]
\end{defn}

Note that the $h$-polynomial has total degree $n$.
If we further assume $\Gamma$ is Eulerian, then we have $h(\Gamma; t, u) = g(\Gamma; t, u)$. 

Suppose $B$ is a rank $n$ Eulerian poset. Then $B\setminus\{\hi\}$ is a rank $n-1$ lower Eulerian poset and its $h$-polynomial is palindromic:
\[
h(B\setminus \{\hi\}; t, u) = \ov{h}(B\setminus \{\hi\}; t, u).
\]
Furthermore we have
\begin{equation}\label{eq:charofh}
(u-t) \cdot h(B\setminus\{\hi\};t,u)=g(B;t,u)-\overline{g}(B;t,u).
\end{equation}
Since $g(B; t, u)$ is defined to be the polynomial of $t$-degree strictly less than $\rk(B)/2$, this formula characterizes the $g$-polynomial. Thus we can compute the $g$-polynomial of an Eulerian poset from 
the $h$-polynomial of its boundary.

\begin{defn} \label{d:localpoly}
	Let $\phi\colon  \Gamma \rightarrow B$ be a strong formal subdivision  between a lower Eulerian poset $\Gamma$ and an Eulerian poset $B$.
	The \emph{local $h$-polynomial} $\ell^h_B(\Gamma;t) \in \mathbb{Z}[t,u]$ is defined by
	\[
		\ell^h_B(\Gamma;t,u) = \sum_{ x \in B } h(\Gamma_x;t,u) \cdot (-1)^{\rk([x, \hi])} g([x,\widehat{1}_B]^*;t,u),
	\]
	where $[x,\widehat{1}_B]^*$ is the \emph{dual poset} of $[x,\widehat{1}_B]$ that has the same underlying set but relations reversed.
\end{defn}

Note that $\ell^h_B(\Gamma; t, u)$ has total degree $\rk(\Gamma)$, and it is an invariant of subdivisions.
Furthermore, we have
\begin{equation}
\label{e:hdecomp}
	h(\Gamma;t,u) = \sum_{ x \in B }  \ell^h_{[\ho_B,x]}(\Gamma_x;t,u) \cdot g([x,\hi_B];t,u).
\end{equation}

\begin{defn}\label{d:mixedpoly}
	Let $\phi\colon  \Gamma \rightarrow B$ be a strong formal subdivision  between a lower  Eulerian poset $\Gamma$ and an Eulerian poset $B$.
	Then the  \emph{mixed $h$-polynomial} $h_B(\Gamma;t',u',t,u) \in \mathbb{Z}[t,t',u,u']$ is defined by
	\[
		h_B(\Gamma;t',u',t,u) =  \sum_{x \in B } (t'u')^{\rk(\Gamma_x)} \cdot 
		\ell^h_{[\ho_B,x]}(\Gamma_x;t/t',u/u') \cdot g([x,\hi_B];tt',uu').
	\]
\end{defn}

Note that $h_B(\Gamma;t',u',t,u)$ is a polynomial since the local $h$-polynomial of $\Gamma_x$ has total degree $\rk(\Gamma_x)$.
The total $(t,u)$-degree and the total $(t',u')$-degree are both equal to the rank of $\Gamma$. By specializing $u$ and $u'$ to $1$ and renaming $t$ to $u$ and $t'$ to $v$, we obtain the mixed $h$-polynomial as in \cite{katz16}

\subsection{Linear maps}

In this subsection we discuss the linear maps defined in \cite{bayer00} that send the $cd$-index to the $g$- and the $h$-polynomials.

Let $\nabla\colon \mathbb{Z}[t,u] \otimes \mathbb{Z}[t,u] \to \mathbb{Z}[t,u]$ be the linear map defined by $\nabla(v\otimes w) = v \cdot w$. We call $\nabla$ the multiplication map.

We define a linear map $\kappa \colon  R^e_{\Psi_{}} \rightarrow \bK[t, u]$ by 
\[\kappa(w) = \begin{cases}
(t-u)^n & \text{if } w= a^n\\
0 & \text{otherwise.}
\end{cases}
\]

For example, we have $\kappa(e) = 0$, $\kappa(1) = 1$ and $\kappa(a) = t-u$. 
We use $\kappa$, the coproduct $\Delta$ and the multiplication $\nabla$ to construct the following linear maps.

\begin{defn}
We define linear maps $F, G\colon  R^e_{\Psi_{}} \rightarrow \bK[t, u]$
recursively as follows:
\begin{enumerate}
\item Let $F(e) =0$ and $G(e) = 1$.

\item For a monomial $w\neq e$, let $F(w)$ be given by 
\[
\ov{F(w)} = (\nabla \circ (G \otimes \kappa) \circ \Delta) (w).\]

\item For a monomial $w \neq e$ of degree $n$, let $G(w)$ to be the sum of monomials in $(u-t) F(w)$ of $t$-degree strictly less than $(n+1)/2$.
\end{enumerate}
\end{defn}

Note that we have $\deg(F(w)) = \deg(w)$ and $\deg(G(w)) = \deg(w) +1$ for a monomial $w$.

For a $cd$-monomial $w$, the image $F(w)$ is palindromic \cite{bayer00}:
\[F(w)= \ov{F(w)}.
\]
Thus it follows from the definition that 
\begin{equation}
(u-t)F(w) = G(w)-\ov{G(w)}.
\end{equation}
This is analogous to \eqref{eq:charofh}, the characterization of the $g$-polynomial.

Here are the images of the $cd$-monomials of degree at most $4$ under the linear maps $F$ and $G$.
\vspace{-2mm}
{\renewcommand{\arraystretch}{1.1}
\begin{table}[H]
\centering
\begin{tabu} to 0.8\textwidth { X[c]| X[c] | X[c]  }
$w$ & $F(w)$ & $G(w)$ \\ \hline
$e$ & $0$ & $1$\\
$1$ & $1$ & $u$\\
$c$ & $u+t$ & $u^2$ \\
$c^2$ & $u^2+t^2$ & $u^3 - t u^2$\\
$d$ & $tu $ & $t u^2$\\
$c^3$ & $u^3 - tu^2 - t^2u + t^3$ & $u^4 - 2t u^3$\\
$cd$ & $0$ & $0$\\
$d c$ & $t u^2  + t^2u $ & $t u^3$\\
$c^4$ & $u^4 - 2t u^3 - 2t^3 u +t^4$ & $u^5 - 3 t u^4 + 2t^2 u^3$\\
$c^2 d$ & $-t^2 u^2 $ & $-t^2 u^3$\\
$c d c$ & $0$ & $0$\\
$d c^2$ & $t u^3+t^3 u $ & $t u^4 -t^2u^3$\\
$d^2$ & $t^2 u^2$ & $t^2 u^3$\\
\end{tabu}
\vspace{-4mm}
\end{table}}

There are explicit formulas for $F$ and $G$ \cite{bayer00}.
Let $p(n, k)= \binom{n}{k} - \binom{n}{k-1}$.
For any non-negative integer $n$ we define polynomials in $\mathbb{Z}[t, u]$ as follows:
	{\renewcommand{\arraystretch}{1.4}
	\begin{center}
$\begin{tabu}{l c l }
Q_n(t, u) & := & \sum_{k=0}^{\floor{(n-1)/2}} (-1)^k p(n-1, k) t^k u^{n-1-k}; \\\relax
R_n(t, u) & := & Q_n(u, t); \\\relax
T_n(t, u) & := & \begin{cases}(-1)^{(n-1)/2} p(n-1, (n-1)/2) \ (tu)^{(n-1)/2} & \text{if } n \text{ is odd}\\
0 & \text{otherwise}
\end{cases}
 \relax
\end{tabu}$	\end{center}}

The polynomials are fundamental building blocks of the formulas for $F$ and $G$. In fact, we have 
$F(c^k) = t R_k + u Q_k$, 
$G(c^k )=u Q_{k+1}$, 
$F(c^k d) = tu T_{k+1}$ and $G(c^k d) = t u^2 T_{k+1}$.

\begin{prop}[\cite{bayer00}]\label{p:valuesofFandG}
The maps $F$ and $G$ take the following values on $cd$-monomials: 
\[F(c^{k_1} d c^{k_2} d \cdots d c^{k_r} d c^k)
= (tu)^r \cdot (t R_k + u Q_k) \cdot \prod_{j=1}^r T_{k_j+1} 
\]
and
\[G(c^{k_1} d c^{k_2} d \cdots d c^{k_r} d c^k)
= (tu)^r \cdot (u Q_{k+1}) \cdot \prod_{j=1}^r T_{k_j+1},
\]
where we set $t R_0 + u Q_0 =1 $.
\end{prop}

Since $T_n(t, u) = 0$ for even $n$, both $F$ and $G$ vanish if any of the $k_i$'s is odd.

\begin{prop}[\cite{bayer00}]\label{p:GF}
For a lower Eulerian poset $P$, we have $F(\Psi_{P}) = h(P)$. For an Eulerian poset $B$ we have $G(\Psi_{\partial B}) = g(B)$.
\end{prop}

In particular for an Eulerian poset $B$, we have
$F(\Phi_{B}) = h(\partial B)$ and $G(\Phi_{B}) = g(B)$. Since $h(B) = g(B)$ we also have $F(\Phi_{B} \cdot a)  = F(\Psi_B)= G(\Phi_{B})$.

We define a linear map $G^* \colon  R^e_{\Psi_{}} \rightarrow \bK[t, u]$ by sending a monomial $w = w_1 \cdots w_n$ where each $w_i \in \{a, b, e\}$
to 
\[G^*(w_1 \cdots w_n) = (-1)^{n+1} G(w_n \cdots w_1). \]
For example we have $G^*(e) = 1$ and $G^*(1) = -u$. 

\subsection{Mixed maps}

In this subsection we define a linear map that sends the mixed $cd$-index to the mixed $h$-polynomial.

Let $\phi\colon \Gamma\to B$ be a strong formal subdivision of rank $0$, where $\Gamma$ is lower Eulerian and $B$ is Eulerian, both of rank $n$.

\begin{defn}
We define a linear map $H_\Omega$ from $R_{\Omega} \rightarrow \bK[t, u]$ as follows:
$$H_\Omega(w' \cdot w)= 
\begin{cases}
F(w'|_{c'\mapsto c, d'\mapsto d}) & \text{if } w=e \\
G(w'|_{c'\mapsto c, d'\mapsto d} \cdot w) & \text{otherwise,}\\
\end{cases}
$$
where $w'$ is an unprimed monomial and $w$ is a primed monomial.
\end{defn}

For example we have $H_\Omega(c' e) = F(c)$ and $H_\Omega(c' c) = G(c^2)$. Note that $\deg(H_\Omega(w)) = \deg(w)+ 1$ for any monomial $w$.

\begin{prop}\label{p:H} The map $H_\Omega$ takes the mixed $cd$-index to the $h$-polynomial in the following sense:
 \[H_\Omega(\Omega_{\phi}) = h(\Gamma).\]
\end{prop}

\begin{proof}
By definition of the mixed $cd$-index, we have
\[
H_\Omega(\Omega_{\phi})
 = F(\ell^\Phi_{\Gamma} )
+ G\left( \Phi_{{\partial \Gamma}}\right)
 = F(\ell^\Phi_{\Gamma}) + F\left(\Phi_{{\partial \Gamma}} \cdot a\right)
 = F(\ell^\Phi_{\Gamma}+\Phi_{{\partial \Gamma}} \cdot a) = F(\Psi_{\Gamma}),
 \]
where the first equality follows from the fact that $\ell^\Phi_\Gamma \cdot e$ is the only term in $\Omega_\phi$ containing $e$.
Thus by Proposition~\ref{p:GF} we have $F(\Psi_{\Gamma}) = h(\Gamma)$.
\end{proof}

We know from \cite[Proposition~7.12]{bayer00}  that for any $cd$-monomials $u$ and $v$, we have
\begin{align*}
G(u \cdot d \cdot v) & = G(u \cdot d) \cdot G(v)\\
F(u\cdot  d\cdot  v) & = G(u \cdot d) \cdot F(v).
\end{align*}
Thus for any $u'\in R'_\Phi$ and $v\in R_\Omega$ we have
\begin{equation}\label{eq:mapHtimesd}
H(u' \cdot d' \cdot v) = G(u'|_{c' \mapsto c, d'\mapsto d} \cdot d) \cdot H(v).
\end{equation}

\begin{prop}
We define the linear map $L_\Omega$ \  from $R_{\Omega} \rightarrow \bK[t, u]$ by \[L_\Omega:=\nabla \circ (H\otimes G^*) \circ \varrho.\]
Then $L_\Omega$ maps the mixed $cd$-index to the local $h$-polynomial in the following sense:
 \[L_\Omega(\Omega_{\phi}) = \ell^h_B(\Gamma).\]
\end{prop}

Note that $\deg(L_\Omega(w)) = \deg(w)+ 1$ for monomial $w$.

\begin{proof}
By the comodule property, 
\[L_\Omega(\Omega_{\phi}) = 
(\nabla \circ (H_\Omega\otimes G^*) ) 
\left(\sum_{x\in B} \Omega_{\phi_x} \otimes \Phi_{[x, \hi]}\right)
 = \sum_{x\in B} h(\Gamma_x) \cdot  (-1)^{\rk([x, \hi])}  g([x, \hi]^*),
\]
which is the local $h$-polynomial $\ell^h_B(\Gamma)$.
\end{proof}

We write down explicit formula for $L_\Omega$ as follows.
For any non-negative integers $k, l$ we define polynomials in $\mathbb{Z}[t, u]$ as follows:
	{\renewcommand{\arraystretch}{1.4}
	\begin{center}
$\begin{tabu}{l c l }
S_k(t, u) & := & 2 u^2 \sum_{j=1}^{k} (-1)^{k-j+1} Q_{j}  Q_{k-j+1}
; \\\relax
S'_{l, k}(t, u) & := &  2 u^2\sum_{j=1}^{k} (-1)^{k-j+1} Q_{l+j} Q_{k-j+1}; \\ \relax
U_{l, k}(t, u) & := & u^2 Q_{l+2}Q_{k+1} (-1)^{k+1}
+ u^2
Q_{l+1}Q_{k+2} (-1)^{k+2}.
 \relax
\end{tabu}$	\end{center}}
\noindent
Note that we have 
	{\renewcommand{\arraystretch}{1.4}
	\begin{center}
$\begin{tabu}{l c l }
(\nabla \circ (H\otimes G^*) ) (\Delta|_{e\mapsto 0} (c^k)) & = & S_k(t, u)
; \\\relax
(\nabla \circ (H\otimes G^*) ) ((c')^l \cdot \Delta|_{e\mapsto 0} (c^k))
& = & S'_{l, k}(t, u); \\ \relax
(\nabla \circ (H\otimes G^*) )
(c^l \cdot \Delta|_{e\mapsto 0}(d) \cdot c^k )
 & = & U_{l, k}(t, u),
 \relax
\end{tabu}$
	\end{center}}
	\noindent
	where $\Delta|_{e\mapsto 0}(w) := \Delta(w)|_{e\mapsto 0}$ and the multiplications are well-defined since $R_\Phi\otimes R_\Phi$ is both a left and a right $R_\Phi$-module.

\begin{prop}
Let $E ' = (c')^{l_1} d' (c')^{l_2} d' \cdots d' (c')^{l_s} d' (c')^l$ and $E = c^{k_1} d c^{k_2} d \cdots d c^{k_r} d c^k$.
Then the explicit formula for $L_\Omega(E' \cdot E)$ is given by \begin{multline}\label{eq:okiamdoingthis}
L_\Omega(E' \cdot E) = (tu)^{r +s}
\left( \prod_{i=1}^s T_{l_{i}+1}  \right)
\left( \prod_{i=2}^r T_{k_{i}+1} \right) T_{k+1}\cdot \\
\Bigg(
(t R_l + u Q_l) \cdot u Q_{k_1+1} 
(-1)^{k_1 +1}
+ T_{l+k_1+1} \frac{u Q_{k+1}}{T_{k+1}}\\
+
\Big(
S'_{l, k_1}
+
T_{l+k_1+1}
\sum_{i=2}^{r+1}
\frac{S_{k_i}}{T_{k_i+1}}
\Big)
+
\Big(
\frac{1}{tu} \frac{U_{l+k_1, k_2}}{T_{k_2+1}}
+
\frac{1}{tu} T_{l+k_1+1} \sum_{i=2}^{r+1} \frac{U_{k_i, k_{i+1}}}{T_{k_{i}+1} T_{k_{i+1}+1}}\Big)\Bigg)
\end{multline}
\end{prop}

\begin{proof}
By applying $\varrho$ to a word $E' \cdot E$, where $E' = (c')^{l_1} d' (c')^{l_2} d' \cdots d' (c')^{l_s} d' (c')^l$ and $E = c^{k_1} d c^{k_2} d \cdots d c^{k_r} d c^k$, we have
\begin{multline*}
 \varrho(E' \cdot E) = 
E' e\otimes E +E' E \otimes e \\
+ 
E' \cdot \sum_{i=1}^{r+1} c^{k_1} d \dots c^{k_{i-1}} d \cdot 
\Delta|_{e\mapsto 0} (c^k_i) \cdot 
d c^{k_{i+1}} \dots d c^{k_{r+1}}\\
+
E' \cdot \sum_{i=1}^{r} c^{k_1} d \dots c^{k_{i}} \cdot 
\Delta|_{e\mapsto 0} (d) \cdot 
c^{k_{i+1}} d \dots d c^{k_{r+1}},
\end{multline*}
where we let $k_{r+1} = k$.
We then apply $\nabla\circ (H^* \otimes G^*)$ and use Proposition~\ref{p:valuesofFandG} to get the desired result.
\end{proof}

Note that $L_\Omega$ vanishes if any of the $l_i$'s is odd. 

\begin{defn}[$L'_\Omega$ and $G'$]
We define a linear map $L'_\Omega\colon  R_{\Omega}\rightarrow \bK[t, t', u, u']$
by setting $$L'_\Omega(w) = (t' u')^{\deg(w)+1}  L(w)(t /t', u /u').$$
We define a linear map $G'\colon  R_{\Omega}\rightarrow \bK[t, t', u, u']$
by setting $$G'(w) = 
G(w)(t t', u u').$$
\end{defn}

By checking the definitions we have \[L'_\Omega (\Omega_{\phi}) =(t' u')^{\rk(\Gamma)} \ \ell^h_B(\Gamma; t /t', u /u')\]
and
\[G'(\Phi_{B}) = g(B; t t', u u').\]
One can write down the explicit formula for $L'_\Omega$ by modifying \eqref{eq:okiamdoingthis}.

\begin{theorem} \label{t:mixcd}
We define a linear map $H'_\Omega\colon  R_{\Omega}\rightarrow \bK[t, t', u, u']$ by setting 
\[H'_\Omega = \nabla \circ (L'_\Omega \otimes G') \circ \varrho.\]
Then $H'_\Omega$ takes the mixed $cd$-index to the mixed $h$-polynomial:
 \[H'_\Omega(\Omega_{\phi}) = h_B(\Gamma; t', u', t, u).\]
\end{theorem}

\begin{proof}
By the comodule property, 
\begin{align*}
H'_\Omega(\Omega_{\phi})
 & = 
(\nabla \circ (L'_\Omega \otimes G') ) 
\left(
\sum_{x\in B} \Omega_{\phi_x} \otimes \Phi_{[x, \hi]}
\right)
\\
& = \sum_{x\in B} 
L'_\Omega(\Omega_{\phi_x})\cdot 
G'(\Phi_{[x, \hi]})\\
& = \sum_{x\in B} 
(t' u')^{\rk(\Gamma_x)} \cdot 
\ell^h_{[\ho, x]}(\Gamma_x; t/t', u /u') \cdot 
g([x, \hi]; t t', u u'),\end{align*}
which is the mixed $h$-polynomial $h_B(\Gamma;t',u',t,u)$.
\end{proof}

Similar to $L_\Omega$ and $L'_\Omega$, the explicit formula for $H'_\Omega$ can be determined. 
Note that for $E' = (c')^{l_1} d' (c')^{l_2} d' \cdots d' (c')^{l_s} d' (c')^l$ and an unprimed monomial $E$, the image $H'_\Omega(E' \cdot E)$ vanishes if any of the $l_i$'s is odd. 

\section{Examples}\label{sec:examples}

We devote this section to examples of strong formal subdivisions.

The following example illustrates the different behaviors of the local $cd$-index and the local $h$-polynomial, despite their similar decomposition formulas.

\begin{example}\label{exam:additive}
Let $\phi\colon \Gamma\to B $ be a rank $0$ strong formal subdivision and $x_0\in B$. 
Suppose further $\phi^{-1}(x)$ is a singleton unless $x= x_0 \in B$.
The mixed $cd$-index of $\phi$ is given by
\begin{align*}
\Omega_\phi(c', d', c, d, e) & = \sum_{x\in B} \ell^\Phi_{\Gamma_x}(c', d') \cdot \Phi_{[x, \hi]}(c, d, e)\\
& = \Phi_B(c, d) +  \ell^\Phi_{\Gamma_{x_0}}(c', d') \cdot \Phi_{[x_0, \hi]}(c, d, e),
\end{align*}
where we use Proposition~\ref{p:localcdvanishing} to get the last equality. 
By \eqref{e:hdecomp}, the $h$-polynomial of $\Gamma$ is
\[
h(\Gamma;t,u) = h(B;t,u) + \sum_{x\geq x_0} \ell^h_{[\ho, x]}(\Gamma_x;t,u)  \cdot g([x, \hi];t,u).
\]
Recall that the linear map $H_\Omega$ takes the mixed $cd$-index to the $h$-polynomial. In particular, the map $H_\Omega$ takes $\Phi_B$ to $h(B)$ and $\Omega_\phi$ to $h(\Gamma)$.
Thus we have
\[
H_\Omega \left(\ell^\Phi_{\Gamma_{x_0}} \cdot \Phi_{[x_0, \hi]} \right) =\sum_{x\geq x_0} \ell^h_{[\ho, x]}(\Gamma_x)  \cdot g([x, \hi]),
\]
showing that the local $cd$-index is not mapped to the local $h$-polynomial. 
In fact, for example the subdivision in Example~\ref{exam:threesubs} has two non-trivial local $h$-polynomials on the right side of the above equation. 
However, if $x_0 = \hi$, we do have
\[
H_\Omega(\ell^\Phi_{\Gamma} \cdot e)= 
\ell^h_{B}(\Gamma).
\]
\end{example}

\medskip

We now study the barycentric subdivision of a simplex. Recall from Example~\ref{exam:boolean} we denote by $\cB_n$ the poset of subsets of $[n]= \{1, \dots, n\}$.

Given two posets $P_1$ and $P_2$, the \emph{Cartesian product} of $P_1$ and $P_2$ is defined as the poset $P_1\times P_2$ with underlying set $\{(x, y)\, | \, x\in P_1, y\in P_2\}$ such that $(x_1, y_1) \leq (x_2, y_2)$ if $x_1\leq x_2$ in $P_1$ and $y_1\leq y_2$ in $P_2$.
We define the \emph{pyramid} of a poset $P$ as $\Pyr(P) := P\times \cB_1$.

\begin{prop}[{\cite[Proposition~4.2]{ehrenborg98}}]
For an Eulerian poset $B$, the $cd$-index of its pyramid $\mathrm{Pry(B)}$ is given by
\begin{equation}\label{eq:pyramid}
\Phi_{\mathrm{Pyr(B)}} = \frac{1}{2} \left(
\Phi_B \cdot c + c\cdot \Phi_B + \sum_{\ho< x < \hi} \Phi_{[\ho, x]} \cdot d \cdot \Phi_{[x, \hi]}.
\right)
\end{equation}
\end{prop}

For example, the Boolean algebra $\cB_n$ is given by 
$\cB_n = {\Pyr(\cB_{n-1})}$. Using \eqref{eq:pyramid}, we have a recursive formula
\[\Phi_{\cB_n} = \frac{1}{2}
\left( \Phi_{\cB_{n-1}} \cdot c + c \cdot \Phi_{\cB_{n-1}}+
\sum_{i=1}^{n-2} \binom{n-1}{i}
 \Phi_{\cB_i} \cdot d \cdot \Phi_{\cB_{n-1-i}} \right),
\]
since the intervals $[\ho, x]$ and $[x, \hi]$ are both Boolean algebras.

Another example is the barycentric subdivision of an Eulerian poset.
Given a graded poset $P$ with $\ho$, the \emph{barycentric subdivision} of $P$ is defined as 
the poset $\Bary(P)$ with underlying set $\{C \mid C\ \text{ is a chain of $P$ containing } \ho\}$ such that $C\leq C'$ if $C\subset C'$ in $P$.
The barycentric subdivision is related to the pyramid operation because for an Eulerian poset $P$ we have $\Bary(P) = \Pyr(\Bary(\partial P))$, that is, the barycentric subdivision of an Eulerian poset is the pyramid of the barycentric subdivision of its boundary.

\begin{prop}\label{p:localbary}
The local $cd$-index of $\cS_n:=\Bary(\cB_n)$ is given by
\[
\ell^\Phi_{\cS_n} = 
\frac{1}{2} \left( c\cdot \Phi_{\partial \cS_n} 
- \Phi_{\partial \cS_n} \cdot c + 
\sum_{T\subset [n-1], T\neq \emptyset}
f_T(\cB_n) \,  \Phi_{\cB_{|T|}} \cdot d \cdot \Phi_{L_T}
\right),
\]
where 
\[
L_{\{t_1<\dots< t_r\}}:= \partial \cS_{t_1} \times \partial \cS_{t_2 - t_1} \times \dots \times \partial \cS_{{n-t_r}}\]
and $f_T(\cB_n)$ is the number of chains with rank set $T$ in $\cB_n$, which is given by
\[
f_{\{t_1<\dots< t_r\}}(\cB_n) = \frac{n!}{(t_1)!(t_2-t_1)!\dots (n-t_r)!}.
\]
\end{prop}

\begin{proof}
We first apply \eqref{eq:pyramid} to the Eulerian poset $\ov{\ses{\cS_n}} = \Pyr(\ov{\partial {\mathcal{S}}_n})$ and get
\[
\Phi_{\ses{\cS_n}} = 
\frac{1}{2} \left(
\Phi_{\partial S_n} \cdot c + c\cdot \Phi_{\partial S_n} + \sum_{\substack{x\in \ov{\partial \cS_n} \\ \ho< x < \hi}} \Phi_{[\ho, x]} \cdot d \cdot \Phi_{[x, \hi]} \right),
\]
since for an Eulerian poset $B$ we have $\Phi_{B} = \Phi_{\partial B}$.
Now each $x\in \partial \cS_n$ is a chain in $\partial \cB_n$, and $y < x$ in $\partial \cS_n$ is equivalent to $y$ being a subchain of $x$, 
hence we have $[\ho, x]\simeq \cB_{\rho(x)}$.

On the other hand, for $x, y\in \partial \cS_n$, the relation $x\leq y$ is equivalent to $y$ containing $x = \{\ho < x_1 < \dots < x_r\}$ as a subchain.
Then $y$ is uniquely given by the union of chains, one from each interval $[\ho, x_1), [x_1, x_2), \dots, [x_r, \hi)$. This gives a bijective correspondence between the set $[x, \hi) \subset \ov{\partial \cS_n}$ and the set 
\[\partial \cS_{\rho(x_1)} \times \partial \cS_{\rho(x_2) - \rho(x_1)} \times \dots \times \partial \cS_{n - \rho(x_r)}.\]
By checking the poset relations one can verify that they are isomorphic as posets. 

Now, we observe that for $x\in \partial \cS_n$
both $[\ho, x]$ and $[x, \hi)$ depend only on $T:= \{\rho(x_1), \dots, \rho(x_r)\}$, and there are $f_T(\cB_n)$ many such chains in $\cB_n$. Hence we have
\[
\Phi_{\ses{\cS_n}}
 = 
\frac{1}{2} \left( c\cdot \Phi_{\partial \cS_n} 
+ \Phi_{\partial \cS_n} \cdot c + 
\sum_{T\subset [n-1], T\neq \emptyset}
f_T \,  \Phi_{\cB_{|T|}} \cdot d \cdot \Phi_{L_T}
\right),
\]
and by definition the local $cd$-index is given by 
\[
\ell^\Phi_{\cS_n} = \Phi_{\ses{\cS_n}} - \Phi_{{\partial \cS_n}} \cdot c. \qedhere
\]
\end{proof}

The $cd$-index of $\partial \cS_n$ can be computed by applying the decomposition theorem to the map $\sigma_n|_{\partial \cS_n}\colon {\partial \cS_n}\to {\partial \cB_n}$
\[
\Phi_{\partial \cS_n} = \Phi_{\cB_n} +
\sum_{i=1}^{n-1} \binom{n}{i}
\ell^\Phi_{{\mathcal{S}}_{i}}
\cdot 
\Phi_{\cB_{n-i}},
\]
which depends on the local $cd$-index of small barycentric subdivisions. Hence Proposition~\ref{p:localbary} gives a recursive way of computing the $cd$-invariants.

For the corresponding subdivision map $\sigma_n\colon {\mathcal{S}}_n\to \cB_n$, we can write down the mixed $cd$-index
\[
\Omega_{\sigma_n}(c', d', c, d, e)= 
\Phi_{\cB_n}(c, d) + 
\sum_{i=1}^n \binom{n}{i}
\ell^\Phi_{{\mathcal{S}}_{i}}(c', d')
\cdot 
\Phi_{\cB_{n-i}}(c , d, e).
\]

We remark that there are explicit formulas for both the local $h$-polynomial \cite[Proposition~2.4]{stanley92} and the mixed $h$-polynomial \cite[Example~5.9]{katz16} for the subdivision $\sigma_n \colon \mathcal{S}_n \to \cB_n$.

\begin{example}
We compute the $cd$-invariants for $\sigma_n \colon \cS_n\to \cB_n$ for $n=2, 3$ explicitly. Note that $\Bary(\cB_1) = \cB_1$.

The $cd$-index of $\cB_2$, the face poset of an edge, is given by $\Phi_{\cB_2} = c$. Note that $\partial \cB_2 = \partial \cS_2$.
For $T=\{1\}$, we have $f_T(\cB_2) = 2$ and $L_T = \partial \cS_1\times \partial \cS_1  = \partial \cB_1$. Thus the local $cd$-index of $\cS_2$ is given by
\[
\ell^\Phi_{S_2} = \frac{1}{2}\left(c^2 - c^2 + 2 d\right) = d
\]
and the mixed $cd$-index of $\sigma_2\colon \cS_2\to \cB_2$
is given by $\Omega_{\sigma_2}(c', d', c, d, e)
=
c + d' e$.
Using our linear maps the local $h$-polynomial is given by 
$\ell^h_B(\cS_2) = tu$
and the mixed $h$-polynomial is given by
$
h_B(\cS_2) = (u u')^2 +  t t' u u'.$

The $cd$-index of $\cB_3$, the face poset of a triangle, is given by $\Phi_{\cB_3} = c^2 + d$. 
The $cd$-index of $\partial {\mathcal{S}}_3$, the face poset of a hexagon, is given by $c^2 + 4d$.
For $T=\{1\}, \{2\}$, we have $f_T(\cB_3) = 3$ and 
$L_T = \partial \cS_2 \times \partial \cS_1 = \partial \cS_1 \times \partial \cS_2 = \partial \cB_2$. For $T = \{1, 2\}$, we have $f_T(\cB_3) = 6$ and $L_T =\partial  \cS_1 \times \partial \cS_1\times \partial \cS_1 = \partial \cB_1$.
Thus the local $cd$-index of ${\mathcal{S}}_3$ is given by
\[
\ell^\Phi_{{\mathcal{S}}_3} = 
\frac{1}{2} \big(
c (c^2 + 4d) - (c^2 + 4d) c + 3 dc + 3 dc + 6 cd
\big) = 5cd + dc,
\]
and the mixed $cd$-index of $\sigma_3 \colon {\mathcal{S}}_3 \to \cB_3$ is given by
\[
\Omega_{\sigma_3}(c', d', c, d, e)
=
c^2 + d + 3 d' + 5c'd'e + d'c'e.
\]
Using our linear maps the local $h$-polynomial is given by 
\[\ell^h_B(\cS_3) = t u^2 + t^2 u
\]
and the mixed $h$-polynomial is given by
\[
h_B(\cS_3) = 
(u u')^3 + 3 t' t u^2 (u')^2 + t t' u u' (t' u + t u').
\]
\end{example}

\medskip

We introduce a class of subdivisions called cuts, obtained by applying a {\em cutting operation} to a poset. 
The cutting operation generalizes an operation of Stanley \cite[Lemma~2.1]{stanley94}, and is dual to the zipping operation introduced in \cite{reading04}.

\begin{defn}
Let $B$ be an Eulerian poset and $x_0 \in B$ be an element of rank $k$. Let $\Sigma_1$ and $\Sigma_2$ be near-Eulerian lower order ideals in $B$ such that $\Sigma_1 \cup \Sigma_2= [\ho, x_0)$ and 
$I:=\Sigma_1\cap \Sigma_2$ satisfies $I  =  \partial \Sigma_1 = \partial \Sigma_2$.
The \emph{cutting operation} on $x_0$
changes $B$ into the poset
$\Gamma := B\setminus x_0 \cup \{y, x_1, x_2\}$ 
where the relations are given by
\begin{enumerate}
\item $z \leq  z'$ if $z\leq  z'$ in $B$;

\item $I \leq y \leq x_i$ for $i=1,2$;

\item $\Sigma_i \leq x_i$ for $i=1, 2$; and 

\item $x_i \leq z$ for $i=1, 2$ if $x_0 \leq z$ in $B$; 
\end{enumerate}
The corresponding poset map $\phi\colon \Gamma\to B$ , defined by 
\[
\phi(z) = \begin{cases}
x_0 & \text{if } z=y, x_1, x_2\\
z & \text{otherwise,}
\end{cases}
\]
is called a \emph{cut} on $x_0$ along $(\Sigma_1, \Sigma_2, I)$.
\end{defn}

If we assume $I\neq \emptyset$, i.e.~$k\geq 2$, then both $\Sigma_1$ and $\Sigma_2$ are graded of rank $k-1$. 
Hence one can easily check that the lower order ideal $I$ is isomorphic to the boundary of an Eulerian poset, the poset $\Gamma$ is lower Eulerian, and $\phi$ is a strong formal subdivision of rank $0$.
Note that we can apply the cutting operation to a near-Eulerian poset, since every near-Eulerian poset is contained in an Eulerian poset.

\begin{example}
\label{exam:cutting}
The subdivision of an edge by adding an interior vertex is a cut, with $x_0 = \hi$, $\Sigma_1 = \{\ho, 1\}$, $\Sigma_2 = \{\ho, 2\}$ and $I = \{\ho\}$, as shown in the following picture. Hence the local $cd$-index is given by $\ell^{\Phi}_\Gamma = d$. 
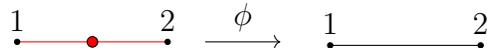
\vspace{-4mm}
\begin{figure}[H]
			\centering
\begin{tikzpicture}[scale=1]
		\draw [color=red] (0,0) -- (1,0) -- (2,0);
		 
		\draw [fill] (0,0) node[above]{$1$} circle [radius = 1pt];
		\draw [fill] (2,0) node[above]{$2$} circle [radius = 1pt];		
		\draw [fill=red] (1,0) circle [radius = 2pt];	
		\draw [->] (2.5,0) -- (3.5,0) node[midway, above] {$\phi$};
\end{tikzpicture}
				\hspace{5pt}
\begin{tikzpicture}[scale=1]
		\draw (0,0) -- (2,0);
		 
		\draw [fill] (0,0) node[above]{$1$} circle [radius = 1pt];
		\draw [fill] (2,0) node[above]{$2$} circle [radius = 1pt];
\end{tikzpicture}
\end{figure}

Similarly, consider a square with $4$ vertices, labeled by $1, 2, 3, 4$. The subdivision of the square into two triangles $\{1, 2, 3\}$ and $\{1, 3, 4\}$ is again a cut, 
with $x_0 = \hi$, $\Sigma_1 = \{\ho, 1, 2, 3, a, b\}$, $\Sigma_2 = \{\ho, 1, 3, 4, c, d\}$ and $I = \{\ho, 1, 3\}$.
The local $cd$-index is given by $\ell^{\Phi}_\Gamma =  c \cdot d$. 
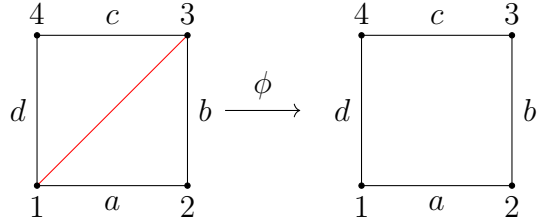
\begin{figure}[H]
			\centering
\begin{tikzpicture}[scale=1]
		\draw (0,0) -- (0,2) node[midway, left]{$d$};
		\draw (0,2) -- (2,2) node[midway, above]{$c$};
		\draw (2,2) -- (2,0) node[midway, right]{$b$};
		\draw (0,0) -- (2,0) node[midway, below]{$a$};
		\draw [color=red] (0,0) -- (2,2);
		
		\draw [fill] (0,0) circle [radius = 1pt] node[anchor=north]{$1$};
		\draw [fill] (2,2) circle [radius = 1pt] node[anchor=south]{$3$};
		\draw [fill] (2,0) circle [radius = 1pt] node[anchor=north]{$2$};
		\draw [fill] (0,2) circle [radius = 1pt] node[anchor=south]{$4$};

		\draw [->] (2.5,1) -- (3.5,1) node[midway, above] {$\phi$};
\end{tikzpicture}
				\hspace{2pt}
\begin{tikzpicture}[scale=1]
		\draw (0,0) -- (0,2) node[midway, left]{$d$};
		\draw (0,2) -- (2,2) node[midway, above]{$c$};
		\draw (2,2) -- (2,0) node[midway, right]{$b$};
		\draw (0,0) -- (2,0) node[midway, below]{$a$};
		 
		\draw [fill] (0,0) circle [radius = 1pt] node[anchor=north]{$1$};
		\draw [fill] (2,2) circle [radius = 1pt] node[anchor=south]{$3$};
		\draw [fill] (2,0) circle [radius = 1pt] node[anchor=north]{$2$};
		\draw [fill] (0,2) circle [radius = 1pt] node[anchor=south]{$4$};								
\end{tikzpicture}
\end{figure}
\end{example}

\begin{lemma}\label{lem:cutschangeincd}
Let $\Gamma$ be a near-Eulerian poset.
Let $\sigma\colon \Pi\to \Gamma$ be a cut on $x_0$ along $(\Sigma_1, \Sigma_2, I)$ where $I\neq \emptyset$.
Then the change in the $cd$-index is given by
\[
\Phi_\Pi - \Phi_\Gamma  = \Phi_{I} \cdot d \cdot \Phi_{[x_0, \infty)}.
\]
\end{lemma}

\begin{proof}
The difference in the flag enumerator is given by
\begin{equation}
\label{e:flagdifference}
\Upsilon_\Pi  - \Upsilon_\Gamma
=
\sum_{\substack{C \text{ a chain of } \Pi \\ 
C \text{ contains $y$, $x_1$, or $x_2$}}}
 u_{\rho(C)}  
  \, \, \, \, \, 
-
\sum_{\substack{C \text{ a chain of } \Gamma \\ C \text{ contains } x_0}} u_{\rho(C)}.
\end{equation}
In the first sum, there are three types of chains: those that contain $y$ only, those that contain $x_i$ only, and those that contain both $y$ and some $x_i$. Hence we have
\begin{align*}
\sum_{\substack{C \text{ a chain of } \Pi \\ 
C \text{ contains $y$, $x_1$, or $x_2$}}} u_{\rho(C)}  
&  = 
 \bigg(
 \Upsilon_I \cdot b a
 +
( \Upsilon_{\Sigma_1} \cdot b
 +
 \Upsilon_{\Sigma_2} \cdot b)
 +
 \Upsilon_I \cdot 2 b b
\bigg) \cdot \Upsilon_{[x_0, \infty)}\\
& = 
 \bigg(
 \Upsilon_I \cdot b a
 +
(\Upsilon_I \cdot a + \Upsilon_{[\ho, x_0)})\cdot b
 +
 \Upsilon_I \cdot 2 b b
\bigg) \cdot \Upsilon_{[x_0, \infty)},
\end{align*}
where the last equality comes from the fact $\Upsilon_{\Sigma_1} + \Upsilon_{\Sigma_2} - \Upsilon_{I} \cdot a  = \Upsilon_{[\ho, x_0)}$.
On the other hand, the second sum in \eqref{e:flagdifference} is given by 
\[
\sum_{\substack{C \text{ a chain of } \Gamma \\ C \text{ contains } x_0}} u_{\rho(C)}
= 
\Upsilon_{[\ho, x_0)} \cdot b \cdot \Upsilon_{[x_0, \infty)}.
\]
Thus the difference becomes
\[
\Upsilon_\Pi  - \Upsilon_\Gamma
=
 \Upsilon_I  \cdot 
(ba + ab + 2bb) 
\cdot \Upsilon_{[x_0, \infty)},
\]
which gives the desired result after a change of variables.
\end{proof}

Note that the change in the $cd$-index does not depend on $\Sigma_1$ and $\Sigma_2$. Hence we only need to specify the element $x_0$ and the lower order ideal $I$ to carry out the computation, as long as a cutting operation along such $I$ exists.
We may also omit $\Sigma_1$ and $\Sigma_2$ when they are clear.
%, e.g.~if $[\ho, x_0)\setminus I$ has two connected components $C_1$ and $C_2$, and each $C_i \cup I$ is near-Eulerian with boundary $I$.

If we have a sequence of cuts, we can use Lemma~\ref{lem:cutschangeincd} to compute the change in the $cd$-index of each cut, and the sum of the changes gives the total change in the $cd$-index.
This is particularly useful when computing the local $cd$-index, see Example~\ref{exam:polygons}, Example~\ref{exam:3polys} and Example~\ref{exam:trig}.

\begin{prop}\label{p:nicecuts}
For a cut $\phi$ along $I\simeq \partial \cB_{k-1}$ with $k>2$, the local $h$-polynomial vanishes and the mixed $h$-polynomial is the same as that of the identity map $\id \colon B\to B$.
\end{prop}

\begin{proof}
Recall we have $G(\Phi_{\cB_n}) = g(\cB_n;t,u) = u^n$. 
From \cite[Proposition~7.11]{bayer00} we have that for any $cd$-monomial of degree $n$
\[
G(v \cdot d) = \begin{cases}
tu T_{=n/2}(G(v)) & \text{if $n$ is even}\\
0 & \text{if $n$ is odd},
\end{cases}
\]
where $T_{=m}(f(t, u))$ is the sum of terms in $f(t, u)$ with $t$-degree exactly $m$. 
In particular, for $k > 2$ we have 
\[
G(\Phi_{\cB_{k-1}} \cdot d) =  0.
\]
Together with \eqref{eq:mapHtimesd}, 
for any $v\in R_\Omega$ and $k>2$ we have
\[
H_\Omega \left(\Phi_{\cB_{k-1}}(c', d') \cdot d' \cdot v\right)= G\left(\Phi_{\cB_{k-1}}(c, d) \cdot d\right)  \cdot H_\Omega(v) 
 =0.
\]
Thus for a cut $\phi$ along $I\simeq \partial \cB_{k-1}$ with $k>2$, by the definitions of $L_\Omega$ and $H'_\Omega$, 
we have
\begin{align*}
L_\Omega(\Omega_\phi) & = L_\Omega(\Phi_B) + H_\Omega(\Phi_{\cB_{k-1}}(c', d') \cdot d') \cdot f_1(t, u) = 0
\\
H'_\Omega(\Omega_\phi) & = H'_\Omega(\Phi_B) + H_\Omega(\Phi_{\cB_{k-1}}(c', d') \cdot d') \cdot f_2(t', u', t, u) = H'_\Omega(\Omega_\id)
\end{align*}
for some polynomials $f_1, f_2$. 
\end{proof}

We now study subdivisions of polygons. 
%First we prove the following lemma.

\begin{lemma}\label{l:factorintocuts}
Given a polygon $B$, any polytopal subdivision $\phi\colon \Gamma \to B$ (as defined in \cite[Example~5.2]{ziegler07}) can be factored into cuts.
\end{lemma}

\begin{proof}
Using the theory of intermediate maps we first decompose $\phi$ into the subdivision of the boundary $\phi_{\partial B}$ and the subdivision of the interior $\phi_{\hi}$. 
Since we have $\phi = \phi_{\hi} \circ \phi_{\partial B}$, it suffices to prove that each subdivision can be factored into cuts.

For $\phi_{\partial B}$, the subdivision is determined by the number of vertices added to the edges.
Note that each addition of a new vertex can be considered as a cutting operation on an edge along $I\simeq \partial \cB_1$, hence $\phi_{\partial B}$ can be factored into cuts.

For $\phi_\hi$, we prove by induction on $f_2$, the number of faces.
If $f_2 = 1$, then $\phi_\hi$ is trivial.
If $f_2=2$, by considering the graph of $\Gamma$ we have that every vertex in the interior is contained by exactly two edges. Thus, this subdivision can be obtained by applying cutting operations on edges along $I\simeq \partial \cB_1$. Undoing all such cutting operations, we arrive at a planar graph with two faces and one edge in the interior, which can be obtained from applying a cutting operation on a face along the endpoints of the interior edge $I \simeq \partial \cB_2$.
This shows that $\phi_\hi$ can be factored into cuts. %Note that we use the planarity of the graph only.

For $f_2>2$, we pick a polytopal path $L$ in $\Gamma$ with only the endpoints contained in the boundary of $\Gamma$.
The path divides $\Gamma$ into two connected component $C_1$ and $C_2$, each with fewer than $f_2$ many faces. Then by induction the maps from the component $C_i$ to $\ov{\partial C_i}$ can be factored into cuts for $i=1, 2$. Together with the case with $f_2 =2$, 
we conclude that $\phi_{\hi}$ can be factored into cuts.
\end{proof}

Note that the proof works for a polytopal subdivision of the boundary of a $3$-polytope.

\begin{example}[Subdivision of polygons]\label{exam:polygons}
Let $B$ be a polygon with $n$ vertices and $\phi\colon \Gamma \to B$ be a polytopal subdivision of $B$.
By Lemma~\ref{l:factorintocuts} the map $\phi$ can be factored into cuts of three types, namely the cut of an edge on the boundary, the cut of an edge in the interior and the cut of a $2$-face.
Each type of cuts changes the triple $(\beta, \gamma, \epsilon)$, where 
{\renewcommand{\arraystretch}{1.4}
	\begin{center}
	$\begin{tabu}{l c l }
\beta &  := &   \# \{x\in \Gamma \, | \, 
\rho_\Gamma(x)= 1, \rho_B(\phi(x)) = 2 \}, \\ \relax
\gamma
& := & \# \{x\in \Gamma \, | \, 
\rho_\Gamma(x)= 1, \rho_B(\phi(x)) = 3 \}, \\ \relax
\epsilon
 & := & \# \{x\in \Gamma \, | \, 
\rho_\Gamma(x)= 2, \rho_B(\phi(x)) = 3 \}.
 \relax
\end{tabu}$
	\end{center}}
\noindent
By looking at the definition of the cutting operation, the cut of an edge on the boundary contributes $1$ to $\beta$, the cut of an edge in the interior contributes $1$ to both $\gamma$ and $\epsilon$, while the cut of a $2$-face contributes $1$ to $\epsilon$.
We  use Lemma~\ref{lem:cutschangeincd} to compute the local $cd$-indices. Together with $\Phi_B = c^2 + (n-2) d$, the mixed $cd$-index of $\phi = \phi_\hi \circ \phi_\Gamma$ is given by
\[
\Omega_{\phi} = c^2 + (n-2) d + \beta d' + 
\gamma d' c' e + (\epsilon - \gamma) c' d' e,
\]
Hence by applying the linear map $H'_\Omega$, we compute the mixed $h$-polynomial 
\[h_B(\Gamma) = 
(u u')^3 + (n-3) t t' u^2 (u')^2
+
\beta t t' u^2 (u')^2+
\gamma t t' u u'
(t' u + t u').
\]
This agrees with the results in \cite[Example~5.7]{katz16}.
Note that the polynomial does not depend on $\epsilon$, since a cut along $I \simeq \partial \cB_2$ does not contribute to the $h$-polynomial.
\end{example}

\begin{example}[Subdivision of the boundary of a $3$-polytope]\label{exam:3polys}
Let $B$ be a $3$-polytope with $f$-vector $(1, f_0, f_1, f_2, 1)$.
It is known that the $cd$-index of $B$ is given by 
\[\Phi_B = c^3 + (f_0 -2 ) dc + (f_2-2) cd.
\]
Let $\phi\colon \Gamma \to B$ be a polytopal subdivision of the boundary of $B$. This implies the poset $\Gamma$ is also Eulerian.
By Lemma~\ref{l:factorintocuts}, we factor $\phi$ into cuts.
By counting the number of cuts of each type, we use Lemma~\ref{lem:cutschangeincd} to compute the mixed $cd$-index of $\phi$ as
\[
\Omega_\phi= 
c^3 + (f_0 -2 ) dc + (f_2-2) cd 
+
\beta d' c + 
\gamma d' c'
+
(\epsilon - \gamma)c' d',
\]
where $\beta$, $\gamma$ and $\epsilon$ are as above.
Then the mixed $h$-polynomial is given by 
\[
h_B(\Gamma) = 
(u u')^4 +(f_0- 4) t t' (u u')^3+
\beta t t' (u u')^3 
+\gamma (t (t')^2 u^3 (u')^2 + t^2 t' u^2 (u')^3 - t^2 (t')^2 u^2 (u')^2),
\]
which is independent of $f_2$ and $\epsilon$.
Note that even though $h_B(\Gamma)$ has negative coefficients, under the specialization $t', u'\mapsto 1$ we have 
$h(\Gamma) = 
u^4 +(f_0- 4) t u^3+
\beta t u^3
+\gamma t  u^3$, the $h$-polynomial of $\Gamma$, which has non-negative coefficients.
\end{example}

\begin{example}[Triangulation of $3$-polytopes]\label{exam:trig}
Let $B$ be a $3$-polytope with $f$-vector $(1, f_0, f_1, f_2, 1)$.
In this example we compute the mixed $cd$-index for any triangulation $\phi \colon \Gamma\to B$
having the following properties: for any $x$ with $\rho(x)=3$, the rank of every element of $\phi^{-1}(x)$ is $2$ or $3$; and every element of $\phi^{-1}(\hi)$ is of rank $3$ or $4$. 
%for any $x\in B$ and $y\in \Gamma$ with $\phi(y) = x$ and $\rho_\Gamma(y) \neq  \rho_B(x)$, then we have $\rho_\Gamma(y)\geq 2$ and $\rho_\Gamma(y)\geq \rho_B(x) -1$.
Note that such triangulation does not exist for every polytope. 
First, we prove $\phi$ factors into cuts. 
By Lemma~\ref{l:factorintocuts}, it suffices to prove this for $\phi_\hi$, the subdivision of the interior.
By hypothesis no edges nor vertices are introduced in the interior. Thus each $2$-face introduced in the interior has boundary isomorphic to $\partial \cB_3$ since $\phi$ is a triangulation. Hence each introduction of a $2$-face in the interior is a cut along $I\simeq \partial \cB_3$, and $\phi_\hi$ can be factored into cuts.

We only use two types of cuts: the cut on a $2$-face on the boundary that adds an edge and a face, and the cut on a $3$-cell in the interior along $I\simeq \partial \cB_3$.
By considering the graph of the boundary of the polytope, there are $2 f_1 - 3 f_2$ many cuts along $I\simeq \partial \cB_2$. 
By considering the number of rank $4$ elements in $\Gamma$, there are $f_3(\Gamma) -1 $ many cuts along $I\simeq \partial \cB_3$.
With $\Phi_{\cB_3} = c^2 +d$, we compute the mixed $cd$-index
\[
\Omega_{\phi} = 
c^3 + (f_0 -2 ) dc + (f_2-2) cd 
+
(2 f_1 - 3 f_2)c' d'
+
(f_3(\Gamma) -1 ) ((c')^2 + d') d' e,
\]
and the mixed $h$-polynomial is given by
\[
h_B(\Gamma) = 
(u u')^4 +(f_0 -4 ) t t' (u u')^3,
\]
which is the mixed $h$-polynomial of the identity map $\id \colon B\to B$, since those cuts do not contribute to the mixed $h$-polynomial.

For example when $P$ is the cube, since the triangulation of the cube into $5$ simplicies 
\cite[Figure~2.61]{deloera10} satisfies the properties, the mixed $cd$-index is given by
\[
\Omega_{\phi} = 
c^3 + 6 dc + 4 cd +6 c' d' + 4 ((c')^2 + d') d' e,\]
and the mixed $h$-polynomial is given by
\[
h_B(\Gamma) = (u u')^4 + 4 t t' (u u')^3.
\]
\end{example}

\begin{remark}
We remark that in \cite{murai14}, Murai and Nevo characterized the $cd$-index of Gorenstein$^*$ posets of rank $5$ using the zipping operation, the dual of the cutting operation.
\end{remark}

\begin{example}\label{exam:threesubs}
Let $B:=\cB_4$ be the simplex with $4$ vertices and $\phi\colon \Gamma\to B$ be the subdivision of $B$ into the triangular bipyramid $\Gamma$ by replacing a facet with the pyramid of its boundary.
The local $cd$-index is given by $\ell^\Phi_\Gamma = 0$, since there is no subdivision in the interior. 
Following Example~\ref{exam:3polys}, we get $\gamma = 1$ and $\epsilon - \gamma  = 2$. Thus the mixed $cd$-index of $\phi$ is given by 
\[\Omega_\phi = c^3 + 2dc + 2cd  + d'c' + 2 c'd'.
\]
By applying $L_\Omega$ and $H'_\Omega$, we get the local $h$-polynomial $\ell^h_B(\Gamma) = - t^2 u^2$ and the mixed $h$-polynomial 
$h_B(\Gamma) = (u u')^4 
+t u^3 (t')^2 (u')^2 
+t^2 u^2 (t') (u')^3
- t^2 u^2 (t')^2 (u')^2$.
Note that both the local $h$-polynomial and the mixed $h$-polynomial have negative coefficients while the mixed $cd$-index does not.
\end{example}

\begin{remark}
The mixed $h$-polynomial of a rational polytopal subdivision is non-negative \cite[Theorem~6.1]{katz16}. Hence the example shows that the extension (as in Definition~\ref{def:extension}) of a polytopal subdivision may not be polytopal.
%A simpler way to see this is that the subdivision of a triangle into a square, which is not a regular subdivision, is given by the extension of the regular subdivision of an edge by adding an interior vertex.
\end{remark}

\begin{remark}
This example was originally given by C. Chan as an example to show that the local $h$-polynomial of a simplicial subdivision may have negative coefficients. It was also studied in \cite[Example~5.6]{katz16}. 
%The example was originally the subdivision of the $3$-simplex into the triangulated triangular bipyramid, which was given by C. Chan as an example to show that the local $h$-polynomial of a simplicial subdivision may have negative coefficients. The example was also studied in \cite[Example~5.6]{katz16}.

\end{remark}

\bibliographystyle{plain}

\end{document}